\documentclass{amsart}

\usepackage{amssymb, amsmath}
\usepackage{esint, mathrsfs}
\usepackage{tikz, xcolor}
\usepackage{ifthen}
\usepackage{enumerate}

\newtheorem{theorem}{Theorem} 
\newtheorem{lemma}[theorem]{Lemma} 
\newtheorem{corollary}[theorem]{Corollary}
\newtheorem{proposition}[theorem]{Proposition}
\newtheorem{assumption}[theorem]{Assumption}

\theoremstyle{definition}
\newtheorem{definition}[theorem]{Definition}

\theoremstyle{remark}
\newtheorem{remark}[theorem]{Remark}

\definecolor{tikzorange}{RGB}{246,99,13}
\definecolor{tikzfuchsia}{RGB}{204,0,102}
\definecolor{tikzblue}{RGB}{0,77,141}
\definecolor{tikzgreen}{RGB}{100,200,0}
\definecolor{tikzpurple}{RGB}{92,44,145}
\definecolor{tikzsky}{RGB}{52,204,253}

\makeatletter
\@namedef{subjclassname@2020}{%
  \textup{2020} Mathematics Subject Classification}
\makeatother

\begin{document}

\title[Maximal Functions for Incompressible Flows]{Construction of Maximal Functions associated with Skewed Cylinders Generated by Incompressible Flows and Applications}

\author[J. Yang]{Jincheng Yang}

\address{Department of Mathematics, The University of Texas at Austin, 2515 Speedway Stop C1200, Austin, TX 78712, USA}

\email{jcyang@math.utexas.edu}


\thanks{\textit{Acknowledgement}. 
The author was partially supported by the NSF grant: DMS-RTG 1840314 (PI: Alexis Vasseur).}

\date{\today}

\begin{abstract}

    We construct a maximal function associated with a family of skewed cylinders. These cylinders, which are defined as tubular neighborhoods of trajectories of a mollified flow, appear in the study of fluid equations such as the Navier-Stokes equations and the Euler equations. We define a maximal function subordinate to these cylinders and show it is of weak type $(1, 1)$ and strong type $(p, p)$ by a covering lemma. As an application, we give an alternative proof for the higher derivatives estimate of smooth solutions to the three-dimensional Navier-Stokes equations. 

    \bigskip


\end{abstract}

\maketitle
\tableofcontents

\newcommand{\grad}{\nabla}
\renewcommand{\div}{\operatorname{div}}
\newcommand{\R}{\mathbb{R}}
\newcommand{\Rd}{\R ^d}
\newcommand{\vp}{\varphi}
\renewcommand*{\d}{\mathop{\kern0pt\mathrm{d}}\!{}}
\newcommand{\e}{\varepsilon}
\newcommand{\rad}{1}
\newcommand{\br}{B _{\rad}}
\newcommand{\be}{B _{\e}}
\newcommand{\Qe}{Q _\e}
\newcommand{\bex}{\be (x)}
\newcommand{\RpRt}{\R _+ \times \R ^3}
\renewcommand{\L}[2]{L ^#1 (#2)}
\newcommand{\tx}[1][]{(t, x \ifthenelse{\equal{#1}{}}{}{; #1})}

\newcommand{\loc}{\mathrm{loc}}
\newcommand{\diam}{\mathrm{diam}}

\newcommand{\pt}{\partial _t}
\newcommand{\La}{\Delta}
\newcommand{\curl}{\operatorname{curl}}
\renewcommand{\div}{\operatorname{div}}
\newcommand{\supp}{\operatorname{supp}}
\newcommand{\Id}{\operatorname{Id}}
\newcommand{\inv}{^{-1}}
\newcommand{\tensor}{\otimes}
\newcommand{\cross}{\times}
\newcommand{\vps}{\vp ^\sharp}
\newcommand{\Pcurl}{\mathbb{P} _{\curl}}
\newcommand{\Pgrad}{\mathbb{P} _{\grad}}
\newcommand{\pq}[1]{L ^{p _#1} _t L ^{q _#1} _x}
\newcommand{\Lp}[1]{L ^{p _#1} _t}
\newcommand{\Lq}[1]{L ^{q _#1} _x}
\newcommand{\LLLH}{L ^\infty _t L ^2 _x \cap L ^2 _t \dot H ^1 _x}
\newcommand{\LH}{L ^2 _t \dot H ^1 _x}
\newcommand{\ind}[1]{\mathbf1 _{#1}}
\newcommand{\inds}[1]{\ind{\{#1\}}}

\newcommand{\Qa}{{Q ^\alpha}}
\newcommand{\ea}{{\e _\alpha}}
\newcommand{\xa}{{x ^\alpha}}
\newcommand{\Xa}{X _{\ea}}
\newcommand{\XA}{X ^\alpha}
\newcommand{\ua}{u _{\ea}}
\newcommand{\ta}{{t ^\alpha}}
\newcommand{\Ta}{{T ^\alpha}}
\newcommand{\Sa}{{S ^\alpha}}
\newcommand{\Ba}{{B ^\alpha}}

\newcommand{\eb}{{\e _\beta}}
\newcommand{\tb}{t ^\beta}
\newcommand{\Tb}{T ^\beta}
\newcommand{\Qb}{Q ^\beta}
\newcommand{\xb}{x ^\beta}
\newcommand{\ub}{u _\eb}
\newcommand{\Xb}{X _\eb}
\newcommand{\XB}{X ^\beta}
\newcommand{\Bb}{B ^\beta}
\newcommand{\Sb}{S ^\beta}
\newcommand{\aj}{\alpha _j}
\newcommand{\bj}{\beta _j}
\newcommand{\eaj}{\e _{\aj}}
\newcommand{\ebj}{\e _{\bj}}
\newcommand{\Sbj}{S ^{\bj}}
\newcommand{\Tbj}{T ^{\bj}}
\newcommand{\Bbj}{B ^{\bj}}
\newcommand{\Qbj}{Q ^{\bj}}

\newcommand{\Qg}{{Q ^\gamma}}
\newcommand{\eg}{{\e _\gamma}}
\newcommand{\xg}{{x _\gamma}}
\newcommand{\Xg}{X _{\eg}}
\newcommand{\ug}{u _{\eg}}
\newcommand{\tg}{{t _\gamma}}
\newcommand{\toxo}[1][]{(t _0, x _0 \ifthenelse{\equal{#1}{}}{}{; #1})}
\newcommand{\tsxs}[1][]{(t _*, x _* \ifthenelse{\equal{#1}{}}{}{; #1})}
\newcommand{\etoxo}{\e _{\tx}}
\newcommand{\Xo}{X ^0}
\newcommand{\Xs}{X ^*}

\newcommand{\psis}[1]{\psi ^{\sharp #1}}
\newcommand{\BMO}{\textit{BMO}}

\newcommand{\vpe}{\vp _\e}
\newcommand{\ue}{u _\e}
\newcommand{\Xe}{X _\e}
\newcommand{\mm}{\mathcal{M}}
\newcommand{\mmu}[1][]{\mm (\grad u \ifthenelse{\equal{#1}{}}{}{(#1)})}
\newcommand{\mmut}{\mmu[t]}
\newcommand{\mmus}{\mmu[s]}
\newcommand{\mmq}{\mm _\mathcal{Q}}

\newcommand{\intR}[1]{\int _{\R ^#1}}
\newcommand{\cci}{C _c ^\infty}
\renewcommand{\emptyset}{\varnothing}

\section{Introduction}

This paper is dedicated to the study of the maximal functions adapted to the Lagrangian description of a flow. When studying the motion of a fluid, there are two different but deeply connected descriptions to work with. The Eulerian formulation records physical quantities such as velocity, temperature, and pressure at fixed positions, while the Lagrangian formulation builds the frame of reference following each moving fluid parcel, and describes their motion and trajectories by a flow map. The transport phenomenon is easier to describe in the Lagrangian formulation, while the diffusion usually suits the Eulerian description better. Let us refer to the works of Constantin (\cite{Constantin2001}), Kukavica and Vicol (\cite{Constantin2016}) for the connection and distinction between these two descriptions in the context of Euler equations. 

For both mathematical study and numerical simulation, sometimes it is necessary to switch between two descriptions. For instance, in computational fluid dynamics, \textit{vortex particle method} treats the fluid as a collection of vortex particles, moving along the trajectories generated by the velocity field, which is in turn recovered from vortex particles. It was early developed by Chorin on the study of the two dimensional Navier-Stokes equations (\cite{Chorin1973}). The validity and convergence of this vortex method in three and two dimensions are confirmed by Beale and Majda in \cite{Beale1982a, Beale1982b}. We refer interested readers to the books of Raviart  (\cite{Raviart1985}), of Cottet and Koumoutsakos (\cite{Cottet2000}) and of Majda and Bertozzi (\cite{Majda2002}) for detailed bibliographies. Majda and Bertozzi also used the particle-trajectory method to show existence and uniqueness results for Euler equations. Even recently, hybrid numerical schemes are still a very active area (\cite{Kaas2013}). To avoid singularities in the computation, a mollification is applied to the velocity field. Therefore, particles are in fact moving along approximated trajectories of this mollified flow defined in Definition \ref{def:flow}. Mollification is also needed for this Lagrangian formulation when the velocity field does not have enough regularity to define trajectories and flow maps, for instance, weak solutions to Navier-Stokes equations or Euler equations. 

Before introducing our new maximal function, let us recall the classical one. For any real-valued or vector-valued function $f \in L ^1 _\loc (\Rd)$ with $d \ge 1$, recall the \textit{classical} maximal function $\mm f$ is defined as
\begin{align}
\label{eqn:spatial-maximal-function}
    (\mm f) (x) := \sup _{r > 0} \fint _{B _r (x)} |f (y)| \d y = \sup _{r > 0} \frac1{|B _r|} \int _{B _r (x)} |f (y)| \d y.
\end{align}
Here $B _r (x)$ is a $d$-dimensional ball with radius $r$ and center $x$, and $|B _r|$ stands for its $d$-dimensional Lebesgue measure $\mathcal L ^d$. Throughout the article, we may use $|\cdot|$ to represent the spatial Lebesgue measure $\mathcal L ^d$ or the spacetime Lebesgue measure $\mathcal L ^{d + 1}$ depending on the context. The strength of the maximal function is that it captures the nonlocal information of a function, in the meantime keeps the homogeneity: it commutes with rigid motion and scaling, as well as scalar multiplication. $\mm$ is a bounded operator on $L ^p $ for $1 < p \le \infty$, and it is also bounded from $L ^1$ to $L ^{1, \infty}$, the weak $L ^1$ space. However, if we include a time variable $t$ in an evolutionary problem, for instance, a transport equation, Euclidean balls in the spacetime are no longer the most natural objects to work with. Instead, we may consider using a spacetime cylinder, or ``skewed cylinder'' transported in the spacetime to be more rigorously defined below. In this paper, we will study such cylinders and construct a maximal function associated with them.

Consider a vector field $u: (S, T) \times \Rd \to \Rd$ satisfying 
$$u \in L ^1 _\loc (S, T; \dot W ^{1, p} (\Rd))$$
for some $1 \le p \le \infty$, where $d \ge 1$ and $-\infty \le S < T \le \infty$ are some finite or infinite initial and terminal time fixed through out this article.
Fix a spatial function $\vp \in \cci (\br)$ satisfying $\int \vp \d x = 1, \vp \ge 0$, where $B _1 \subset \Rd$ is a unit ball of dimension $d$. Define the usual mollifier function
$
    \vpe := \e^{-d} \vp (\cdot/\e) \in \cci (B _\e).
$
We denote a universal constant by $C$ if it depends only on $\vp$ and $d$. Its value may change from line to line. We define the \textit{spatially} mollified velocity $\ue: (S, T) \times \Rd \to \Rd$ by
\begin{align*}
    \ue \tx := [u (t, \cdot) * \vpe] (x) 
    &= \intR d u (t, x - y) \vpe (y) \d y.
\end{align*}
By convolution, $\ue \in L ^1 _\loc (S, T; C ^1 (\Rd))$. Let us now give the definition for the mollified flow and the skewed cylinders.

\begin{definition}
[Mollified Flow, Skewed Cylinders]
\label{def:flow}
For some fixed $\e > 0$ and $\tx \in (S, T) \times \Rd$, define the \textbf{mollified flow} $\Xe (t, x; \cdot)$ to be the unique solution to the following initial value problem
\begin{align*}
    \begin{cases}
        \dot \Xe \tx[s] = \ue (s, \Xe \tx[s]) \\
        \Xe \tx[t] = x
    \end{cases} \qquad s \in (S, T)
\end{align*}
where the dot means to take derivative in the last argument $s$. Moreover, if $S + \e ^2 < t < T - \e ^2$,
define the \textbf{skewed parabolic \footnote{
Parabolic scaling---$\e ^2$ in time versus $\e$ in space---will not be indispensable in this paper. We only employ it because of its applications to the Navier-Stokes equations, but all the results can be generalized to other time-space scaling.} cylinder} with center $\tx$ and radius $\e$ by
\begin{align*}
    Q _\e \tx := \left\lbrace
        (s, y) : |s - t| < \e ^2, |y - \Xe (t, x; s)| < \e
    \right\rbrace.
\end{align*}
\end{definition}

Heuristically speaking, skewed cylinders defined in Definition \ref{def:flow} are objects appearing in the Lagrangian formulation but written in Eulerian coordinates. Indeed, they are following the mollified flow and capturing particles that are close to the center trajectories. Similar to the difficulty of bridging these two formulations, the difficulty of working with these cylinders comes from the lack of control on the distortion. Without a uniform control on the velocity field, these skewed cylinders following different flows may include nonuniform geometric properties. Despite this technical challenge, the maximal function will provide us a tool for overcoming this conceptual difficulty. Instead of taking the average in balls, now we construct a new maximal function that takes the average in the skewed cylinders that are ``admissible''.

\begin{definition}[Admissibility, Maximal Function]
    \label{def:admissible-cylinder}
    Given $\e > 0$, $x \in \Rd$, $t \in (S + \e ^2, T - \e ^2)$, we define a skewed cylinder $\Qe \tx$ by Definition \ref{def:flow}. For $\eta > 0$, we say $\Qe \tx$ is \textbf{$\eta$-admissible} if
    \begin{align} 
        \label{eqn:initial-assumption}
        \e ^2 \fint _{\Qe \tx} \mmus (y) \d y \d s
        =
        \frac{1}{\e ^d |Q _1|} \int _{\Qe \tx} \mmu \d y \d s < \eta.
    \end{align}    
    Here $\mm$ is the spatial-only maximal function defined in \eqref{eqn:spatial-maximal-function}, and with a slight abuse of notation, we also use $|Q _1|$ to represent the $(d+1)$-dimensional space-time Lebesgue measure $\mathcal L ^{d + 1}$ of a cylinder with radius 1.
    % Let $\mathcal Q _\eta$ denote the set of $\eta$-admissible cylinders (see Definition \ref{def:admissible-cylinder}). 
    % Define the maximal operator $\mmq$ with respect to this collection of cylinders as the following. 
    For any locally integrable function $f \in L ^1 _\loc ((S, T) \times \Rd)$, for every $\tx \in (S, T) \times \Rd$ we define a new maximal function $\mmq$ by
    \begin{align*}
        \mmq (f) \tx := \sup _{\e > 0} \left\lbrace
            \fint _{\Qe \tx} |f (s, y)| \d y \d s
            : \Qe \tx \text{ is $\eta$-admissible}
        \right\rbrace.
    \end{align*}
\end{definition}

Note that in the sup we actually need $\e ^2 < \min\{t - S, T - t\}$ to define $\Qe \tx$, and we will justify in Section \ref{sec:harmonic} that admissible choices of $\e$ exist for almost every $\tx$, so that $\mmq$ is well-defined.

The main result of this paper is the following.

\begin{theorem}
\label{thm:maximal-function}
Let $\eta < \eta _0$ for some small universal constant $\eta _0 > 0$. If $u$ is divergence-free, and $\mm(\grad u) \in L ^p ((S, T) \times \Rd)$ for some $1 \le p \le \infty$
\footnote{In the case $1 < p \le \infty$, since $\mm$ is a bounded operator on $L ^p (\Rd)$, this condition is equivalent to $\grad u \in L ^p ((S, T) \times \Rd)$.}, then $\mmq$ associated with $\eta$-admissible cylinders generated by $u$ satisfies the following.

\begin{enumerate}[\upshape (1)]
    \item $\mmq$ is of strong type $(\infty, \infty)$, i.e. for $f \in L ^\infty ((S, T) \times \Rd)$, it holds that
    \begin{align*}
        \| \mmq f \| _{L ^\infty ((S, T) \times \Rd)} \le \| f \| _{L ^\infty ((S, T) \times \Rd)}.
    \end{align*}
    
    \item $\mmq$ is of weak type $(1, 1)$, i.e. for $f \in L ^1 ((S, T) \times \Rd)$, $\lambda > 0$, the Lebesgue measure of the superlevel set satisfies
    \begin{align*}
        \mathcal L ^{d + 1} \left(\left\lbrace
            \tx \in (S, T) \times \Rd: (\mmq f) \tx > \lambda
        \right\rbrace\right) \le \frac {C _1} \lambda \| f \| _{L ^1 ((S, T) \times \Rd)}.
    \end{align*}
    
    \item $\mmq$ is of strong type $(q, q)$ for any $1 < q < \infty$, i.e. for $f \in L ^q ((S, T) \times \Rd)$, it holds that
    \begin{align*}
        \| \mmq f \| _{L ^q ((S, T) \times \Rd)} \le C _q \| f \| _{L ^q ((S, T) \times \Rd)}.
    \end{align*}
\end{enumerate}
\end{theorem}

Let us now explain why we are interested in these skewed cylinders and the maximal function related to them. In many scaling-invariant partial differential equations, it is a common technique to zoom in near a point and conduct a local analysis in its neighborhood, and use this obtained local information to deduce global results. This form of argument usually consists of two parts: one is a local theorem, which handles the rescaled problem near a point, and the second is a local-to-global step, which contributes to some global information. For instance, the 3D Navier-Stokes equations
\begin{align}
    \label{eqn:navier-stokes}
    \pt u + u \cdot \grad u + \grad P = \La u, \qquad \div u = 0
\end{align}
are scaling invariant. In particular, $u _\lambda$ and $P _\lambda$ defined by
\begin{align*}
    u _\lambda (t, x) = \lambda u (\lambda ^2 t, \lambda x), \qquad P _\lambda (t, x) = \lambda ^2 P (\lambda ^2 t, \lambda x)
\end{align*}
are also solutions to \eqref{eqn:navier-stokes}. In \cite{Caffarelli1982}, Caffarelli, Kohn and Nirenberg investigated the partial regularity of suitable weak solutions to the Navier-Stokes equations by zooming into a so-called \textit{parabolic cylinder}, where \textit{parabolic} refers to the fact that the spatial scale is $\lambda$ while the temporal scale is $\lambda ^2$. They showed that if a suitable solution $u$ satisfies 
\begin{align*}
    \limsup _{r \to 0} \frac1r \int _{t - \frac78r^2} ^{t + \frac18r^2} \int _{B _r (x)} |\grad u (s, y)| ^2 \d y \d s \le \eta
\end{align*}
for some fixed small $\eta$, then $u$ is regular at $(t, x)$. From this local theorem, they used a covering argument to conclude a global result, that the parabolic measure $\mathscr{P} ^1$ of the singular set is zero. This was an improvement from Scheffer's result (\cite{Scheffer1980}) which stated the singular set has at most Hausdorff dimension $\frac53$. The reason for this improvement is that $\iint |\grad u| ^2 \d x \d t$ has a stronger scaling than other quantities, which is $\iint |\grad u _\lambda| ^2 \d x \d t = \frac1\lambda \iint |\grad u| ^2 \d x \d t$.

Quantitative global results can also follow from this kind of scaling arguments. Choi and Vasseur (\cite{Vasseur2010, Choi2014}) estimated higher derivatives, by locally controlling higher derivatives using De Giorgi technique applied to quantities with the same strong scaling as  $\iint |\grad u| ^2$. In particular, one must avoid using $\iint |u| ^\frac{10}3$, which has a weaker scaling. However, without controlling the flux, the parabolic regularization cannot overcome the nonlinearity. A natural idea would be to utilize the Galilean invariance of Navier-Stokes equations and work in a neighborhood following the flow. Instead of working on parabolic cylinders, they worked on skewed parabolic cylinders as we defined above.

The advantage of using such skewed cylinders is that, by taking out the mean velocity, one can use velocity gradient to control the velocity in the local study. The maximal function associated with these skewed cylinders then will help us better bridge the local study to global results. 

Let us mention that a similar construction also appears in the recent development of convex integration for Euler equations by Isett (\cite{Isett2017, Isett2018}) and the subsequent work of Isett and Oh (\cite{Isett2016}), where the authors call the mollified flow \textit{coarse scale flow} and skewed cylinders $u _\e$\textit{-adapted Eulerian cylinders}. The difference from the previous definition is that their apertures of mollification, radii of cylinder bases, and lengths of time spans are chosen differently from here. The purpose is however the same, which is to kill the mean velocity, and to obtain estimates that are dimensionally correct. 

Note that Theorem \ref{thm:maximal-function} has already been used in \cite[Corollary 1]{Vasseur2020} to show the following result.

\begin{theorem}
    Let $u$ be a suitable weak solution to the 3D Navier-Stokes equations \eqref{eqn:navier-stokes} with initial data $u |_{t = 0} = u _0 \in L ^2 (\R ^3)$. For any $q > \frac43$, $K \subset \subset (0, T) \times \R ^3$, there exists a constant $C _{q, K}$ depending on $q$ and $K$ such that the following holds,
    \begin{align*}
        \| \grad ^2 u \| _{L ^{\frac43,q} (K)} \le C _{q, K} 
        \left( 
            \| u _0 \| _{L ^2 (\R ^3)} ^\frac32 + 1
        \right).
    \end{align*}
\end{theorem}

This is an improvement of \cite{Constantin1990} where the result was shown for $L ^q$ with $q < \frac43$, and of \cite{Lions1996} where it was shown for $L ^{\frac43, \infty}$.

In this paper, we provide a first application of Theorem \ref{thm:maximal-function} to give an alternative proof for the results of Choi and Vasseur in \cite{Choi2014}, as an example of using the maximal function to obtain global results from local estimates.

\begin{theorem}
\label{thm:vasseur}
Let $(u, P)$ be a smooth solution to \eqref{eqn:navier-stokes} in $(0, T)$ with initial data $u _0 \in L ^2$, let $d \ge 1$, $\alpha \in [0, 2)$, denote $f = |(-\La) ^\frac\alpha2 \grad ^d u|$, $p = \frac4{d+1+\alpha}$. We have
\begin{align*}
    \left\| f \inds{f ^p > C _{d, \alpha} t ^{-2}} \right\| _{L ^{p, \infty} ((0, T) \times \R ^3)} ^p \le C \| u _0 \| _{L ^2 (\R ^3)} ^2 .
\end{align*}
\end{theorem}

This paper is organized as follows. Bounds on the maximal function rely on a Vitali-type covering lemma, which is introduced in Section \ref{sec:covering-lemma}, where we define admissible cylinders and prove the covering lemma for them. We use this covering lemma to show some properties of the maximal function in Section \ref{sec:harmonic}. Finally, in Section \ref{sec:navier-stokes} we use the maximal function to give an alternative proof for the higher derivative estimates for the Navier-Stokes equations. 

%%%%%%%%%%%%%%%%%%%%%%%%%%%%%%%%%%%%%%%%%%%%%%%%
%%%%%%%%%%%%%%%%%%%%%%%%%%%%%%%%%%%%%%%%%%%%%%%%
%%%%%%%%%%%%%%%%%%%%%%%%%%%%%%%%%%%%%%%%%%%%%%%%
%%%%%%%%%     Covering Lemma Part      %%%%%%%%%
%%%%%%%%%%%%%%%%%%%%%%%%%%%%%%%%%%%%%%%%%%%%%%%%
%%%%%%%%%%%%%%%%%%%%%%%%%%%%%%%%%%%%%%%%%%%%%%%%
%%%%%%%%%%%%%%%%%%%%%%%%%%%%%%%%%%%%%%%%%%%%%%%%

\section{Covering Lemma}
\label{sec:covering-lemma}

In this section, we derive some basic properties of the mollified flows and admissible cylinders, then use them to prove the covering lemma.
% In conclusion, the effect of $\e$ on $\Qe$ is two folds: $\e$ indicates the size and time-span of $\Qe$; $\e$ determines the aperture of mollification of $\ue$, which determines center flow $\Xe$ and decides the direction of $\Qe$.

\subsection{Preliminaries}

We first note the following easy pointwise estimate on the mollified velocity gradient.
\begin{lemma}[Pointwise Estimate on $\grad \ue$] 
For $\tx \in (S, T) \times \Rd$, $y \in \Rd$, and $\e, r > 0$, we have 
\begin{align}
    % \label{original-u-estimate}
    % |\ue \tx| &
    % \le C \e ^{-\frac32} \| u (t) \| _{\L2{\bex}}, \\
    \label{original-grad-estimate}
    |\grad \ue \tx | &
    \le C \e ^{-d} \| \grad u (t) \| _{\L1{\bex}}, \\
    % \label{original-grad-estimate-by-u}
    % |\grad \ue \tx| &
    % \le C \e ^{-\frac52} \| u (t) \| _{\L2{\bex}}, \\
    \label{maximal-grad-estimate}
    |\grad \ue (t, x)| &
    \le C \e ^{-d} \left(\frac{|y - x|}{\e} + 2\right) ^d \|\mmut \| _{\L1{\be (y)}}, \\
    \label{maximal-grad-estimate-differente}
    |\grad \ue (t, x)| &
    \le C \e ^{-d} \left(\frac{|y - x| + r + \e}{r} \right) ^d \|\mmut \| _{\L1{B _r (y)}}.
\end{align}
\end{lemma}

\begin{proof}
The first estimate follows easily from the scaling that
\begin{align*}
    % |\ue \tx| &\le \|u (t, x + \e \cdot) \| _{\L2{\br}} \| \vp \| _{\L2{\br}} 
    % \le C \e ^{-\frac32} \| u (t) \| _{\L2{\bex}}, \\
    % |\grad \ue\tx | &\le \|\grad _x u (t, x + \e \cdot) \| _{\L1{\br}} \| \vp \| _{\L\infty{\br}} 
    % \le C \e ^{-3} \| \grad u (t) \| _{\L1{\bex}}.
    \grad \ue \tx %= (\grad u *_x \vpe) \tx 
    &= \intR d \grad u (t, x - y) \vpe (y) \d y \le 
    \| \grad u (t) \| _{\L1{\bex}}
    \| \vpe \| _{L ^\infty}.
\end{align*}
% For the third one, we use
% \begin{align}
%     \ue \tx &= \int _{\RpRt} u (t, x + \e y) \vp (y) \d y \\
%     &= \int _{\RpRt} u (t, \e y) \vp \left(y - \frac{x}{\e}\right) \d y, \\
%     \grad \ue \tx &= \int _{\RpRt} u (t, \e y) \grad _x \vp \left(y - \frac{x}{\e}\right) \d y \\
%     &= -\frac1\e \int _{\RpRt} u (t, \e y) \grad \vp \left(y - \frac{x}{\e}\right) \d y \\
%     &= -\frac1\e \int _{\RpRt} u (t, x + \e y) \grad \vp \left(y\right) \d y,
% \end{align}
% Therefore
% \begin{align}
%     |\grad \ue\tx | &\le \e ^{-1} \|u (t, x + \e \cdot) \| _{\L2{\br}} \| \grad \vp \| _{\L\infty{\br}} 
%     \le C \e ^{-\frac52} \| u (t) \| _{\L2{\bex}}.
% \end{align}
This indicates that by controlling the average of $\grad u$ in a small ball $\bex$, we can control the size of mollified gradient at the center $x$. To control the mollified gradient from elsewhere, we need a maximal function to gather nonlocal information. For any $x' \in \bex$, $y' \in B _r (y)$, we have $|y' - x'| \le |y - x| + r + \e =: K r$, so $\bex \subset B _{K r} (y ')$ and the integral of $\grad u$ can be bounded by
\begin{align*}
    \int _{\bex} |\grad u (t, z)| \d z
    \le \int _{B _{K r} (y')} |\grad u (t, z)| \d z 
    &= \left|B _{K r} (y')\right| \fint _{B _{K r} (y')} |\grad u (t, z)| \d z \\
    &\le K ^d |B _r| \mmut (y').
\end{align*}
Since the above holds for any $y' \in B _r (y)$, by taking the average of right-hand side in $B _r (y)$ we have
\begin{align}
    \notag
    \| \grad u (t) \| _{\L1{\bex}}
    &\le K ^d |B _r| \fint _{B _r (y)} \mmut (y') \d y' \\
    \label{maximal-grad-estimate-pf}
    &= K ^d \|\mmut \| _{\L1{B _r (y)}}.
\end{align}
This bound and estimate \eqref{original-grad-estimate} yield the third estimate, and the second estimate is a special case of the third when $r = \e$.
\end{proof}

As can be seen here, $\mmu$ controls how mollified velocities alter in space. This observation motivates us to introduce the notion of admissibility in Definition \ref{def:admissible-cylinder}. Let us provide a heuristic explanation for the choice of homogeneity in \eqref{eqn:initial-assumption}. Consider two skewed cylinders, both with radius of order $\e$, starting at the same time with distance also of order $\e$. If $\grad u$ is of order $\e ^{-2} \eta$, then their velocities roughly differ by $\e ^{-1} \eta$, so in a time span of length $\e ^2$, they at most diverge $\e \eta$ further away, so their distance will remain of order $\e$. This ensures cylinders do not deviate relatively too far away, and will be crucial in the covering lemma.

\begin{remark}
\label{rem:jensen}
For $1 < p < \infty$, \eqref{eqn:initial-assumption} is weaker than the $L^p$ analogue
\begin{align*}
    \e ^2 \left(
        \fint _{\Qe \tx} \mm (| \grad u| ^p) \d y \d s
    \right) ^\frac1p < \eta.
\end{align*}
This is because Jensen's inequality implies that
\begin{align*}
    \left(
        \fint _{\Qe \tx} \mmu \d y \d s
    \right) ^p 
    \le 
    \fint _{\Qe \tx} [\mmu]^p \d y \d s
\end{align*}
and
\begin{align*}
    [\mmu] ^p (x) &= \sup _{r > 0} \left(
        \fint _{B _r (x)} |\grad u| \d y
    \right) ^p 
    \le \sup _{r > 0} 
        \fint _{B _r (x)} |\grad u| ^p \d y
    = [\mm (| \grad u| ^p)] (x).
\end{align*}
\end{remark}

Next, let us discuss the trajectories of the mollified flow that pass through an admissible cylinder. 

\begin{lemma} 
\label{lem:xxp}
There exists a universal constant $\eta _1 > 0$ such that the following is true. 
Given $\e > 0$, $t _0 \in (S + \e ^2, T - \e ^2)$ and $x _0 \in \Rd$, suppose $\Qe \toxo$ is $\eta$-admissible as defined in Definition \ref{def:admissible-cylinder} with $\eta < \eta _1$.
For any $\tsxs \in \Qe \toxo$, we have
\begin{align} 
    \label{eqn:uniform-closeness}
    |\Xe \tsxs[t] - \Xe \toxo[t]| \le 2 \e
    % |\Xe \tx[s] - \Xe \toxo[s]| \le 2 \e
\end{align}
at any given time $t \in (t _0 - \e ^2, t _0 + \e ^2)$.
% for any $\tsxs \in \Qe \toxo$, and any $t \in (t _0 - \e ^2, t _0 + \e ^2)$.
% provided $\eta$ is small enough.
\end{lemma}

\begin{proof}
% Denote $\Delta X (t) = \Xa(t _0, x _0; t) - \XA(t)$, then
To ease the notation, we denote 
\begin{align*}
    \Xs (t) := \Xe \tsxs[t], \quad \Xo (t) := \Xe \toxo[t], \quad \Delta X (t) := \Xs (t) - \Xo (t),
\end{align*}
thus we need to show $|\Delta X (t)| \le 2 \e$. We argue by contradiction and suppose $|\Delta X (s _*)| > 2 \e$ at some $s _* \in (\Sa, \Ta)$. Without loss of generality, suppose $s _* > t _*$. Note that
\begin{align*}
    |\Delta X (t _*)| = |\Xs (t _*) - \Xo (t _*)| = |x _* - \Xe \toxo[t _*]| < \e < 2 \e
\end{align*}
because $\tsxs \in \Qe \toxo$. Since $\Delta X$ is absolute continuous, there must exist an $r _* \in (t _*, s _*)$ such that 
\begin{align}
    \label{eqn:continuity-assumption}
    |\Delta X (t)| \le 2 \e \text{ for any } t \in [t _*, r _*], \qquad |\Delta X (r _*)| = 2 \e.
\end{align}
% Now suppose $|\Delta X| \le 2 \e$ in $[s _1, t _1]$ for some $s _1 \in (t - \e ^2, t _0)$ and $t _1 \in (t _0, t + \e ^2)$, then
%  so $x _0 \in \Ba (t _0)$, thus \eqref{eqn:uniform-closeness} is true at $t = t _0$. For other $t \in (\Sa, \Ta)$,
For almost every $t \in [t _*, r _*]$, the growth rate of the difference $\La X$ can be bounded by 
\begin{align*}
    \frac{\d}{\d t} \left| \Delta X (t) \right| \le \left| \frac{\d}{\d t} \Delta X (s) \right| 
    &= \left| \dot X ^* (t) - \dot X ^0 (t) \right| \\
    &= \left| \ue (t, \Xs (t)) - \ue (t, \Xo (t)) \right| \\
    &\le |\grad \ue (t, \xi _t)| |\Delta X (t)|
\end{align*}
for some $\xi _t$ between $\Xs (t)$ and $\Xo (t)$. We can bound the gradient term by
\begin{align*}
    |\grad \ue (t, \xi _t)|
    &\le C \e ^{-d} \left(\frac{|\xi _t - \Xo (t)|}{\e} + 2\right) ^d \| \mmut \| _{\L1{B _\e (\Xo (t))}} \\
    &\le C \e ^{-d} \left(\frac{|\Delta X (t)|}{\e} + 2\right) ^d \| \mmut \| _{\L1{B _\e (\Xo (t)}}
\end{align*}
% where
% \begin{align*}
%     & |\ue (\tau, \Xo (\tau)) - \ue (\tau, \XA (\tau))| \\
%     & \qquad \le |\grad \ue (\tau, \xi _\tau)| |\Delta X (\tau)| \\
%     &\qquad \le C \ea^{-d} \left(\frac{|\xi _\tau - \XA (\tau)|}{\e} + 2\right) ^d \|\mmutau \| _{\L1{\Ba (\tau)}}  |\Delta X (\tau)| \\
%     &\qquad \le C \left(\frac{|\Delta X (\tau)|}{\e} + 2\right) ^d \frac{\| \mmutau \| _{\L1{\Ba (\tau)}}}{\ea^d}  |\Delta X (\tau)|.
% \end{align*}
using \eqref{maximal-grad-estimate} for $x = \xi _t$ and $y = \Xo (t)$. By \eqref{eqn:continuity-assumption}, $|\Delta X (t)| \le 2 \e$ for any $t \in [t _*, r _*]$, so in the above coefficient $C (\frac{|\Delta X (t)|}{\e} + 2) ^d \le C (2 + 2) ^d = C$, thus for almost every $t \in [t _*, r _*]$ we have
\begin{align*}
    \frac{\d}{\d t} \left| \Delta X (t) \right| \le \frac{C}{\e ^d} \| \mmut \| _{\L1{B _\e (\Xo (t))}} | \Delta X (t)|.
\end{align*}
% \begin{align*}
%     |\Delta X (s)| \le |\Delta X (t _0)| + \int _{t _0} ^s \frac{C}{\ea^d} \| \mmutau \| _{\L1{\Ba (\tau)}} |\Delta X (\tau)| \d \tau.
    % & \ue (\tau, \Xe (t _0, x _0; \tau)) - \ue (\tau, X (\tau)) \\
    % &\qquad \le C \left(2 + 2\right) ^d \frac{\| \mmutau \| _{\L1{B _\e (X (\tau))}}}{\e ^d}  |\Delta X (s)| \\
    % &\qquad = C \frac{\| \mmutau \| _{\L1{B _\e (\tau)}}}{\e ^d}  |\Delta X (s)|.
% \end{align*}
By Gr\"onwall's inequality, we reach a conclusion that
\begin{align*}
    |\Delta X (r _*)| & \le |\Delta X (t _*)|
    \exp \left( 
        \int _{t _*} ^ {r _*}
        \frac{C}{\e ^d} \| \mmut \| _{\L1{B _\e (\Xo (t))}} \d t
    \right) \\
    & \le \e
    \exp \left( 
        \int _{t _0 - \e ^2} ^ {t _0 + \e ^2}
        \frac{C}{\e ^d} \| \mmut \| _{\L1{B _\e (\Xo (t))}} \d t
    \right) \\
    & = \e
    \exp \left( 
        \frac{C}{\e ^d} \| \mmu \| _{\L1{\Qe \toxo}}
    \right) \\
    & \le \e \exp \left(C \eta\right)
\end{align*}
which contradicts \eqref{eqn:continuity-assumption} when choosing $\eta < \eta _1 = \frac1C \log 2$.
\end{proof}

To conclude this subsection, we discuss two streamlines with different $\e$ that start from the same location. Before that, we introduce some notations. Let $\alpha$ be an index. Given $\ea > 0$, $\ta \in (S + \ea ^2, T - \ea ^2)$, $\xa \in \Rd$, we abbreviate
\begin{align}
    \label{eqn:notation}
    \begin{aligned}
        \XA (t) &:= \Xa (\ta, \xa; t), \qquad 
        \Ba (t) := B _{\ea} \left(\XA (t)\right) \subset \Rd, \\
        \Sa &:= \ta - \ea ^2, \qquad \qquad \quad \;\,
        \Ta := \ta + \ea ^2, \\
        Q ^\alpha &:= Q _{\ea} (\ta, \xa) = \left\lbrace
        \tx: \Sa < t < \Ta , x \in B ^\alpha (t)
        \right\rbrace.
    \end{aligned}
\end{align}
For $\lambda > 0$, we denote the spatial dilation of a cylinder $\Qa$ by
\begin{align}
    \label{eqn:lambda-qa}
    \lambda \Qa &:= \left\lbrace
    \tx: \Sa < t < \Ta, x \in \lambda \Ba (t) = B _{\lambda \ea} \left(\XA(t)\right)
    \right\rbrace.
\end{align}
Notice that different from upright cylinders or cubes, for $\e _1 < \e _2$, it is not known that $Q _{\e _1} \tx \subset Q _{\e _2} \tx$, because their center streamlines $X _{\e _{1,2}}$ solve different equations. As we will see later, this lack of monotonicity only poses a minor technical difficulty. For the same reason, note that $\lambda \Qe \tx \neq Q _{\lambda \e} \tx$, and neither is necessarily contained in the other. 

\begin{lemma} 
\label{lem:eep}
Recall that $\eta _1$ is a universal constant defined in Lemma \ref{lem:xxp}. There exists a universal constant $\eta _0 < \eta _1$ such that the following is true. Given $\ea > \frac12 \eb > 0$, $\ta \in (S + \ea ^2, T - \ea ^2)$, $\tb \in (S + \eb ^2, T - \eb ^2)$ and $\xa, \xb \in \Rd$, suppose $\Qa = Q _\ea (\ta, \xa)$, $\Qb = Q _\eb (\tb, \xb)$ are $\eta$-admissible as defined in Definition \ref{def:admissible-cylinder} with $\eta < \eta _0$.
% Let be two $\eta$-admissible cylinders defined in Definition \ref{def:admissible-cylinder} with $\eb < 2 \ea$. 
For any $\tsxs \in \Qa \cap \Qb$, we have
% If $\toxo \in \Qa \cap \Qb$, then
\begin{align} 
    \label{eqn:conclusion-2}
    |\Xb \tsxs[t] - \Xa \tsxs[t]| \le \ea
\end{align}
at any given time $t \in (\Sa, \Ta) \cap (\Sb, \Tb)$.
\end{lemma}

\begin{proof}
    Denote 
    \begin{align*}
        X ^1 (t) = \Xa \tsxs[t], \qquad X ^2 (t) = \Xb \tsxs[t], \qquad \Delta X (t) = X ^1 (t) - X ^2 (t),
    \end{align*}
    thus we need to show $|\Delta X (t)| \le \ea$. Note that
    \begin{align*}
        \Delta X (t _*) = X ^2 (t _*) - X ^1 (t _*) = \Xb \tsxs[t _*] - \Xa \tsxs[t _*] = x _* - x _* = 0.
    \end{align*}
    Similar as in the last lemma, we argue by contradiction and suppose there exists $r _* \in (t _*, \min \{\Ta, \Tb\})$, such that
    \begin{align}
        \label{eqn:continuity-assumption-2}
        |\Delta X (t)| \le \ea \text{ for any } t \in [t _*, r _*], \qquad |\Delta X (r _*)| = \ea.
    \end{align}
    For almost every $t \in [t _*, r _*]$, the time derivative of $\La X$ is calculated as
    \begin{align}
        \frac{\d}{\d t} \Delta X (t) 
        \notag
        &= \dot X ^2 (t) - \dot X ^1 (t) \\
        \notag
        &= \ub (t, X ^2 (t)) - \ua (s, X ^1 (t)) \\
        \notag
        &= \ub (t, X ^2 (t)) - \ua (s, X ^2 (t)) + \ua (s, X ^2 (t)) - \ua (s, X ^1 (t)) \\
        \label{eqn:double-integral}
        &= \int _{\ea} ^{\eb} \frac{\partial}{\partial \e} \ue (t, X ^2 (t)) \d \e + \ua (t, X ^2 (t)) - \ua (t, X ^1 (t)).
    \end{align}
% \begingroup
% \allowdisplaybreaks
% \begin{align}
%     \notag
%     & \Xb (t _0, x _0; t) - \Xa (t _0, x _0; t) \\
%     \notag
%     &\qquad = \int _{t _0} ^t \dot X _{\eb} (t _0, x _0; s) - \dot X _{\ea} (t _0, x _0; s) \d s \\
%     \notag
%     &\qquad = \int _{t _0} ^t \ub (s, \Xb (t _0, x _0; s)) - \ua (s, \Xa (t _0, x _0; s)) \d s \\
%     \notag
%     &\qquad = \int _{t _0} ^t \ub (s, \Xb (t _0, x _0; s)) - \ua (s, \Xb (t _0, x _0; s)) \d s \\
%     \notag
%     &\qquad \qquad + \int _{t _0} ^t \ua (s, \Xb (t _0, x _0; s)) - \ua (s, \Xa (t _0, x _0; s)) \d s \\
%     \label{eqn:double-integral}
%     &\qquad = \int _{t _0} ^t \int _\ea ^\eb \frac{\partial}{\partial \e} \ue (s, \Xb (t _0, x _0; s)) \d\e \d s \\
%     \notag
%     &\qquad \qquad + \int _{t _0} ^t \ua (s, \Xb (t _0, x _0; s)) - \ua (s, \Xa (t _0, x _0; s)) \d s.
% \end{align}
% \endgroup
We will use $\Qb$ to control the first integral term and use $\Qa$ to control the rest. Note that 
\[
    \frac{\partial}{\partial \e} u _\e \tx = \frac{\partial}{\partial \e} \int _{\Rd} u (t, x - \e y) \vp (y) \d y
    = \int _{\Rd} \grad _x u (t, x - \e y) \cdot -y \vp (y) \d y,
\]
thus we can control its absolute value by
\begin{align}
    \notag
    \left| 
        \frac{\partial}{\partial \e} u _\e \tx
    \right|
    & \le \e^{-d} \| \grad u (t) \| _{\L1{\bex}} 
    \| y \vp (y) \| _{L^\infty} \\
    \notag
    & = C \e ^{-d} \| \grad u (t) \| _{\L1{\bex}} \\
    \notag
    & \le C \e ^{-d} \left(\frac{\left|x - \XB(t)\right| + \eb + \e}{\eb} \right) ^d \|\mmut \| _{\L1{B _{\eb} (\XB (t))}} \\
    \label{eqn:partial-e-estimate}
    & = C \eb ^{-d} \left(\frac{\left|x - \XB (t)\right| + \eb + \e}{\e}\right) ^d \|\mmut \| _{\L1{\Bb(t)}}.
\end{align}
Here we use \eqref{maximal-grad-estimate-pf} with $r = \eb$ and $y = \XB (t)$ in the last inequality to control $\| \grad u (t) \| _{\L1{\bex}}$. Thanks to Lemma \ref{lem:xxp}, $|X ^2 (t) - \XB (t)| \le 2 \eb$. Since $\e$ is between $\eb$ and $\ea > \frac12 \eb$, we have
\begin{align*}
    \frac{\left| X ^2 (t) - \XB (t) \right| + \eb + \e}{\e} \le \frac{3\eb + \e}{\e} \le 7.
\end{align*}
Hence if we set $x = X ^2 (t)$ in \eqref{eqn:partial-e-estimate}, we would get 
\[ 
    \left| \frac{\partial}{\partial \e} u _\e (t, X ^2 (t)) \right|
    \le C \eb ^{-d} \| \mmut \| _{\L1{\Bb(t)}}, 
\] 
thus we can bounded the $\partial _\e$ term in \eqref{eqn:double-integral} by
\begin{align*}
    \left| 
        \int _{\ea} ^{\eb} \frac{\partial}{\partial \e} \ue (t, X ^2 (t)) \d \e 
    \right| 
    &\le C |\eb - \ea| \eb ^{-d}\|\mmut \| _{\L1{\Bb(t)}} \\
    &\le C \ea \eb ^{-d}\|\mmut \| _{\L1{\Bb(t)}}.
\end{align*}
The remaining terms in \eqref{eqn:double-integral} can be bounded similar as in Lemma \ref{lem:xxp} as
\begin{align*}
    |\ua (t, X ^2 (t)) - \ua (t, X ^1 (t))| 
    & \le |\grad \ua (t, \xi _t)| |\Delta X (t)| \\
    & \le C \ea ^{-d} \| \mmut \| _{\L1{\Ba (t)}} |\Delta X (t)|.
\end{align*}
Combining these two bounds in \eqref{eqn:double-integral}, for almost every $t \in [t _*, r _*]$, the growth rate of $\La X$ is bounded by
\begin{align*}
    \frac{\d}{\d t} |\Delta X (t)| &\le C \left(
        \eb ^{-d}\|\mmut \| _{\L1{\Bb(t)}} + \ea ^{-d} \| \mmut \| _{\L1{\Ba (t)}}
    \right) \\
    & \qquad \times \left(
        \ea + | \Delta X (t)|
    \right).
\end{align*}
By Gr\"onwall's inequality, we would reach
\begin{align}
    \notag
    \ea + |\Delta X (r _*)| &\le \left(\ea + |\Delta X (t _*)|\right) \exp \bigg(
        C \int _{t _*} ^{r _*} 
        \eb ^{-d}\|\mmut \| _{\L1{\Bb(t)}} \\
        \notag
        & \qquad \qquad \qquad \qquad \qquad \qquad 
        + \ea ^{-d} \| \mmut \| _{\L1{\Ba (t)}} \d t
    \bigg) \\
    \label{eqn:inclusion}
    &\le \ea \exp ( 2 C \eta )
\end{align}
which contradicts \eqref{eqn:continuity-assumption-2} when choosing $\eta < \eta _0 = \min\{\eta _1, \frac1{2C} \log 2\}$. 
\end{proof}

\subsection{Covering Lemma for Admissible Cylinders}

The goal of this section is to prove a Vitali-type covering lemma for $\eta$-admissible cylinders, provided $\eta < \eta _0$. The key ingredient is Proposition \ref{prop:closeness}, which shows that if two cylinders intersect, then during their shared life span, they are uniformly close to each other. 
% Heuristically speaking, the reason this can be done is that $\grad u$ is controlled by \eqref{eqn:initial-assumption}, so the velocity at nearby points cannot differ too much.
% To show this result, we need two lemmas. 
% Recall that 
Based on this proposition, we conclude in Lemma \ref{lem:Qa-control-Qas} that for an $\eta$-admissible cylinder $\Qa$, the union of all $\eta$-admissible cylinders with comparable or less radius that intersect $\Qa$ has a comparable total measure as $\Qa$. The covering lemma will be a consequence of Lemma \ref{lem:Qa-control-Qas}.

Throughout this subsection, we employ the notations introduced in \eqref{eqn:notation}.

\begin{proposition}
\label{prop:closeness}
For any pair of intersecting $\eta$-admissible cylinders $\Qa, \Qb$ as in \eqref{eqn:notation} with $\eb < 2 \ea$ and $\eta < \eta _0$ chosen in Lemma \ref{lem:eep}, at any $t \in (\Sa, \Ta) \cap (\Sb, \Tb)$, we have $\Bb (t) \subset 9 \Ba (t)$.
\end{proposition} 

That is, if $\Qa$ intersects $\Qb$ with $\eb < 2 \ea$, then $\Qb \cap {\{\Sa < t < \Ta\}} \subset 9 \Qa$. Recall that $\lambda \Qa$ is the spatial dilation defined in \eqref{eqn:lambda-qa}. The proof is based on Lemma \ref{lem:xxp} and Lemma \ref{lem:eep} which control the trajectories at the level of $\Qe$. See Figure \ref{fig:my_label} for our strategy. 

\begin{figure}[ht!]
    \centering
        \begin{tikzpicture}
          %%%%% Xea and Xeb %%%%%
          \draw [tikzblue, ultra thick, line cap=round] (-0.8, -0.6) to [out=0,in=200] (1.4, 0.1);
          \draw [tikzblue, ultra thick, line cap=round] (-0.8, -0.5) to [out=-20,in=190] (1.4, -0.6);
          \draw [tikzblue, fill] (-0.52, -0.6) circle [radius=0.1];
          \draw [tikzorange, fill] (1.4, 0.1) circle [radius=0.05];
          \draw [tikzfuchsia, fill] (1.4, -0.6) circle [radius=0.05];
          %%%%% Q ^a %%%%%
          \draw [tikzorange, ultra thick, line cap=round, fill=tikzorange, fill opacity=0.2] (-1.8, -1.3) -- (-1.8, 0.7) to [out=10,in=225] (1.4, 2.3) -- (1.4, 0.3) to [out=225,in=10] (-1.8, -1.3);
          \draw [tikzorange, line width=1mm, line cap=round] (-1.78, -0.3) to [out=10,in=225] (1.38, 1.3);
          \draw [tikzorange, fill] (-0.2, 0.15) circle [radius=0.1];
          \draw [-{Bar}, tikzorange, thick, line cap=round] (-1, 2.5) -- (-1.8, 2.5);
          \draw [-{Bar}, tikzorange, thick, line cap=round] (0.6, 2.5) -- (1.4, 2.5);
          %%%%% Q ^b %%%%%
          \draw [tikzfuchsia, ultra thick, line cap=round, fill=tikzfuchsia, fill opacity=0.2] (-0.8, -1.6) -- (-0.8, -0.4) to [out=-20,in=150] (1.7, -1) -- (1.7, -2.2) to [out=150,in=-20] (-0.8, -1.6);
          \draw [tikzfuchsia, line width=1mm, line cap=round] (-0.78, -1) to [out=-20,in=150] (1.68, -1.6);
          \draw [tikzfuchsia, fill] (0.4, -1.22) circle [radius=0.1];
          \draw [-{Bar}, tikzfuchsia, thick, line cap=round] (-0.4, -2.55) -- (-0.8, -2.55);
          \draw [-{Bar}, tikzfuchsia, thick, line cap=round] (1.2, -2.55) -- (1.7, -2.55);
          %%%%% Notations %%%%%
          \node at (-0.2, 2.5) {\color{tikzorange} $(\Sa, \Ta)$};
          \node at (0.4, -2.55) {\color{tikzfuchsia} $(\Sb, \Tb)$};
          \node at (-1.45, 0.35) {\color{tikzorange} $\Qa$};
          \node at (-0.5, -1.4) {\color{tikzfuchsia} $\Qb$};
          \node at (-0.4, 0.5) {\color{tikzorange} \tiny $(\ta, \xa)$};
          \node at (0.4, -1.5) {\color{tikzfuchsia} \tiny $(\tb, \xb)$};
          \node at (-1.3, -0.6) {\color{tikzblue} \tiny $\tsxs$};
          \node at (2, 1.3) {\color{tikzorange} $\XA (t)$};
          \node at (2.5, 0.1) {\color{tikzblue} $\Xa \tsxs[t]$};
          \node at (2.5, -0.6) {\color{tikzblue} $\Xb \tsxs[t]$};
          \node at (2.3, -1.6) {\color{tikzfuchsia} $\XB (t)$};
        \end{tikzpicture}
    
    \caption{$\Qa$ and $\Qb$ intersect}
    \label{fig:my_label}
\end{figure}

\begin{proof}
Let $\eta _0$ be chosen as in Lemma \ref{lem:eep}. Fix some $\tsxs \in \Qa \cap \Qb$. For any $\tx \in \Qb$ with $\Sa < t < \Ta$, we apply the triangle inequality to estimate 
\begin{align*}
    |x - \XA (t)| &\le |x - \XB (t)| \\
    & \qquad + |\XB (t) - \Xb \tsxs[t]| \\
    & \qquad + |\Xb \tsxs[t] - \Xa \tsxs[t]| \\
    & \qquad + |\Xa \tsxs[t] - \XA (t)| \\
    & \le \eb + 2 \eb + \ea + 2 \ea.
\end{align*}
Here the first term is because $x \in \Bb (t)$, the second and the fourth are due to Lemma \ref{lem:xxp}, and the third term is controlled by Lemma \ref{lem:eep}. Since $\eb < 2 \ea$, we have
\begin{align*}
    |x - \XA (t)| < 9 \ea.
\end{align*}
\end{proof}

We remark here that if we take a sharper estimate in each step of Lemma \ref{lem:xxp} and Lemma \ref{lem:eep} (and require a smaller $\eta$), the factor 9 can be easily improved to $5 + \delta$ for any $\delta > 0$. 5 is also the factor that appeared in the original Vitali covering lemma for balls. Recall that in the proof of Vitali lemma, an important reason why we get a comparable volume is because if two balls $B _{r _1} (x _1) \cap B _{r _2} (x _2) \neq \emptyset$ with $r _2 < 2 r _1$, then $B _{r _2} (x _2) \subset 5B _{r _1} (x _1)$. Unfortunately, this geometric property cannot be realized in our case, because an admissible cylinder with \eqref{eqn:initial-assumption} has no control on the past and the future velocities. As a consequence, it is unlikely to cover $\Qb$ by a dilation of $\Qa$ in space-time. However, this requirement can be relaxed as the following. See Section 1.1 of \cite{Stein1993} for a more general setting.

\newcommand{\Qas}{Q ^\alpha _*}
\begin{lemma}
\label{lem:Qa-control-Qas}
Given a fixed $\Qa$ and a family of $\{\Qb\} _{\beta \subset \Lambda}$ as in \eqref{eqn:notation} such that for each $\Qb$, $\Qa \cap \Qb \neq \emptyset$, $\eb < 2 \ea$, and they are $\eta$-admissible for $\eta < \eta _0$. Let $\Qas = \bigcup _{\beta \in \Lambda} \Qb$ denote the union of this family. Then there exists a universal constant $C$ such that
\begin{align*}
    |\Qas| \le C |\Qa|.
\end{align*}
\end{lemma}

\newcommand{\calQ}[1]{\mathcal{Q} ^{(#1)}}
\newcommand{\calQi}[1]{\mathcal{Q} ^{(#1)} _i}
\begin{proof}
Without loss of generality, we may assume that $\{\Qb\} _{\beta \subset \Lambda}$ is a finite collection. The general case can be proven using the finite case. Note that each $\Qb$ is an open set. For any compact subset $K \subset \subset \Qas$, $K$ admits a finite open cover, thus $|K| \le C |\Qa|$ using the finite case. Since the inequality holds for any compact subset $K$, it must also be true for $\Qas$.

% Denote $T _+ := \Ta + \ea ^2$, and $T _- := \Ta - \ea ^2$. Then $\Ta = (T _-, T _+)$, and f
For each $\Qb$, we can break it into $\Qb = \Qb _+ \cup \Qb _- \cup \Qb _\circ$, where
\begin{align*}
	\Qb _+ &= \Qb \cap \{t \ge \Ta\}, \\
	\Qb _- &= \Qb \cap \{t \le \Sa\}, \\
	\Qb _\circ &= \Qb \cap \{\Sa < t < \Ta\}.
\end{align*}
From Proposition \ref{prop:closeness}, we can conclude that
\begin{align}
    \label{eqn:middle-control}
	\bigcup \nolimits_{\beta \in \Lambda} \Qb _\circ \subset 9 \Qa \Rightarrow \left| \bigcup \nolimits_{\beta \in \Lambda} \Qb _\circ \right| \le 9^d |\Qa|.
\end{align}
As mentioned in the remark, we cannot bound the size of $ \bigcup _{\beta \in \Lambda} \Qb _+$ or $\bigcup _{\beta \in \Lambda} \Qb _-$ directly by $\Qa$, as their center streamlines can diverge away from $\XA$ after $\Ta$. Fortunately, we do not need them to be close to $\XA$, as long as we can show they remain a small distance to each other.

Let us measure $\bigcup _{\beta \in \Lambda} \Qb _+$. First, we group the cylinders by their radii. Denote
\begin{align}
    \label{eqn:defn-Lambda-i}
    \Lambda _i = \{
        \beta \in \Lambda: 2 ^{-i} \ea \le \eb < 2 ^{-i + 1} \ea
    \}.
\end{align}
Because each $\eb < 2 \ea$, we have $\Lambda = \bigcup _{i \in \mathbb{N}} \Lambda _i$, hence we can write the union as 
\begin{align}
    \label{eqn:decompose-1st}
    \bigcup \nolimits_{\beta \in \Lambda} \Qb _+ 
    = \bigcup \nolimits_{i \in \mathbb{N}} \bigcup \nolimits_{\beta \in \Lambda _i} \Qb _+.
\end{align}

Now we fix $i$ and estimate the size of $\bigcup _{\beta \in \Lambda _i} \Qb _+$. Clearly we can disregard the empty ones, and assume $\Ta < \Tb$ for each $\beta \in \Lambda _i$. 
% Denote the terminal time $\Tb _+ = \Tb + \eb ^2$ (which is greater than $T _+$). 
To begin with, set $\calQi0 = \{ \Qb _+ \} _{\beta \in \Lambda _i}$. Then we repeat the following two steps: at the $j$-th iteration ($j \ge 1$),
\begin{enumerate}[{\ttfamily Step 1.}]
    \item Select some $\bj$ such that $\Tbj = \max \left\lbrace \Tb : \Qb _+ \in \calQi{j - 1} \right\rbrace$.
    \item From $\calQi{j - 1}$ we remove any $\Qb _+$ such that $\Bb (\Ta) \cap \Bbj (\Ta) \neq \emptyset$, and denote the rest by $\calQi j$.
\end{enumerate} 
After finitely many iterations, $\calQi{n + 1}$ will be empty, and we have a list of $Q ^{\beta _1} _+$, ..., $Q ^{\beta _n} _+$. We claim that 
\begin{align}
    \label{eqn:decompose-2nd}
    \bigcup \nolimits_{\beta \in \Lambda _i} \Qb _+ \subset \bigcup _{j = 1} ^n 9 \Qbj _+ .    
\end{align}
To see why this is true, take any $\Qb _+ \in \calQi0$. It must have been removed from $\calQi {j - 1}$ at some step $j$ in the above process. This implies $\Bb (\Ta) \cap \Bbj (\Ta) \neq \emptyset$, and $\Tb \le \Tbj$. Also, we have $\eb < 2 \e _{\bj}$, which is actually true for any pair of cylinders by our selection of $\Lambda _i$ according to \eqref{eqn:defn-Lambda-i}. Therefore, by Proposition \ref{prop:closeness} we have $\Bb (t) \subset 9 \Bbj (t)$ at any $t \in (\Sb, \Tb) \cap (\Sbj, \Tbj)$. Because $\Sb, \Sbj \le \Ta \le \Tb \le \Tbj$, we have $\Qb _+ \subset 9 \Qbj _+$ and this proves the claim \eqref{eqn:decompose-2nd}.

Note that by our construction, $\{\Bbj (\Ta)\} _{j = 1} ^n$ are pairwise disjoint, and they are all inside $9 B ^\alpha (\Ta)$ by the Proposition \ref{prop:closeness}. Therefore their total measure is
\begin{align*}
    \sum _{j = 1} ^n |\Qbj _+| 
    &\le \sum _{j = 1} ^n |\Bbj (\Ta)| \cdot 2 (\ebj) ^2 \\
    &\le \sum _{j = 1} ^n 2 \cdot |\Bbj (\Ta)| \cdot (2 ^{-i + 1} \ea) ^2 \\
    & = 2 \cdot 4 ^{-i + 1} {\ea} ^2 \left|
        \bigcup \nolimits _{j = 1} ^n \Bbj (\Ta)
    \right| \\
    &\le 2 \cdot 4 ^{-i + 1} {\ea} ^2 |9 \Ba (\Ta)| = 4 ^{-i + 1} \cdot 9 ^d |\Qa|.
\end{align*}
Combining with the claim \eqref{eqn:decompose-2nd}, we have
\begin{align*}
    \left|
        \bigcup \nolimits _{\beta \in \Lambda _i} \Qb _+
    \right| \le \left|
        \bigcup \nolimits _{j = 1} ^n 9 \Qbj _+
    \right| \le 9 ^d \sum _{j = 1} ^n |\Qbj _+| \le 4 ^{-i + 1} \cdot 9 ^{2d} |\Qa|.
\end{align*}
Finally, take the summation over $i$, and \eqref{eqn:decompose-1st} yields
\begin{align*}
    \left|
        \bigcup \nolimits _{\beta \in \Lambda} \Qb _+
    \right| \le \sum _{i = 0} ^\infty \left|
        \bigcup \nolimits _{\beta \in \Lambda _i} \Qb _+
    \right| \le \sum _{i = 0} ^\infty 4 ^{-i + 1} \cdot 9 ^{2d} |\Qa| = \frac{16}3 \cdot 9^{2d} |\Qa|.
\end{align*}
The same proof also applies to $\bigcup \nolimits _{\beta \in \Lambda} \Qb _-$. Therefore, together with estimate \eqref{eqn:middle-control}, we have proven that
\begin{align*}
    |Q ^\alpha _*| = \left|
        \bigcup \nolimits _{\beta \in \Lambda} \Qb
    \right| 
    & \le \left|
        \bigcup \nolimits _{\beta \in \Lambda} \Qb _+
    \right| + \left|
        \bigcup \nolimits _{\beta \in \Lambda} \Qb _-
    \right| + \left|
        \bigcup \nolimits _{\beta \in \Lambda} Q ^\beta _\circ
    \right| \\
    & \le \frac{16}3 \cdot 9^{2d} |\Qa| + \frac{16}3 \cdot 9 ^{2d} |\Qa| + 9 ^d |\Qa| = C |\Qa|.
\end{align*}
\end{proof}

We are finally ready to show the Vitali-type covering lemma.

\newcommand{\Qaj}{Q ^{\alpha _j}}
\begin{proposition}[Covering Lemma]
\label{prop:covering-lemma}
Let $\mathcal{A}$ be an index set and let
\begin{align*}
    \mathcal{Q} = \{ \Qa = Q _{\e _\alpha} (\ta, \xa) : \alpha \in \mathcal{A} \}    
\end{align*}
be a collection of $\eta$-admissible cylinders, where $\eta < \eta _0$ defined in Lemma \ref{lem:eep} and $\e _\alpha$ are uniformly bounded.
% , $t _\alpha \in [S, T]$ are uniformly bounded, and $|x _\alpha| < R$ are also uniformly bounded. 
Then there is a pairwise disjoint sub-collection $\mathcal{P} = \{Q ^{\alpha _1}, Q ^{\alpha _2}, \dots, Q ^{\alpha _n}, \dots \}$ (finite or infinite) such that 
\begin{align*}
    \sum \nolimits_j \left|
        \Qaj
    \right| \ge \frac1C \left|
        \bigcup \nolimits _{\alpha \in \mathcal{A}} \Qa
    \right|,
\end{align*}
where $C$ is a universal constant.
\end{proposition}

\begin{proof}
% Again we may assume $\mathcal{A}$ is finite, otherwise we can choose a sub-collection with a measure greater then half of the original one.
With the help of the previous lemma, the proof of the covering lemma is the same as the classical one in \cite{stein1970}.
We select the sub-collection $\mathcal P$ by the following procedure. To begin with, set $\calQ0 = \mathcal{Q}$. Then repeat the following two steps: at the $j$-th iteration ($j \ge 1$),
\begin{enumerate}[{\ttfamily Step 1.}]
    \item Select some $\alpha _j$ such that $\eaj > \frac12 \sup _{\Qa \in \calQ{j - 1}} \{\ea \}$.
    \item From $\calQ{j - 1}$ we remove any $\Qa$ that intersects with $\Qaj$, and denote the rest by $\calQ j$.
\end{enumerate} 
This procedure may stop after a certain step if $\calQ{n + 1} = \emptyset$, or it can continue indefinitely.
% because there are finitely many cylinders in $\calQ0$, and in each iteration we get rid of at least one cylinder. 
We denote the chosen ones by $\mathcal{P} := \{Q ^{\alpha _1}, \dots, Q ^{\alpha _n}, \dots \}$ (finite or infinite). They are pairwise disjoint due to our strategy.

Suppose that $\sum _j \left|\Qaj\right| < \infty$, otherwise the conclusion is automatically true. Thus either $\mathcal P$ is a finite collection, or $\mathcal P$ is infinite and $\eaj \to 0$ as $j \to \infty$. In either case, each $\Qa$ must be removed from $\calQ j$ at some iteration. Otherwise, we would have $\Qa \in \calQ{j - 1}$ for all $j$, then {\ttfamily Step 1} would imply that $\eaj > \frac12\ea$ for all $j$, thus the sequence $\eaj$ cannot converges to zero. Now suppose $\Qa \in \calQ{j - 1} \setminus \calQ j$, then we have $\Qa \cap \Qaj \neq \emptyset$, and $\ea < 2\e _{\alpha _j}$. This implies
\begin{align*}
    \Qa \subset \Qaj _* := \bigcup _{\alpha \in \mathcal{A}} \left\{
        \Qa \in \mathcal{Q}: \ea < 2 \e _{\alpha _j}, \Qa \cap \Qaj \neq \emptyset
    \right\}.
\end{align*}
Thus $\bigcup _{\alpha \in \mathcal{A}} \Qa \subset \bigcup _{j = 1} ^n \Qaj _*$, and finally we control the measure of the union by
\begin{align*}
    \left|
        \bigcup \nolimits _{\alpha \in \mathcal{A}} \Qa
    \right| \le 
    \left|
        \bigcup \nolimits _{j = 1} ^n \Qaj _*
    \right| \le \sum _{j = 1} ^n |\Qaj _*| \le C \sum _{j = 1} ^n |\Qaj|
\end{align*}
thanks to Lemma \ref{lem:Qa-control-Qas}.
\end{proof}

% \begin{remark}
% Up to here, the only property that we used about $u$ is that $u \in \L\infty{0, T; \L2{\R ^3}} \cap \L2{0, T; \dot H ^1 (\R ^3)}$. The same arguments can be generalized to other flows that are not divergence free.
% \end{remark}

%%%%%%%%%%%%%%%%%%%%%%%%%%%%%%%%%%%%%%%%%%%%%%%%
%%%%%%%%%%%%%%%%%%%%%%%%%%%%%%%%%%%%%%%%%%%%%%%%
%%%%%%%%%%%%%%%%%%%%%%%%%%%%%%%%%%%%%%%%%%%%%%%%
%%%%%%%%%    Harmonic Analysis Part    %%%%%%%%%
%%%%%%%%%%%%%%%%%%%%%%%%%%%%%%%%%%%%%%%%%%%%%%%%
%%%%%%%%%%%%%%%%%%%%%%%%%%%%%%%%%%%%%%%%%%%%%%%%
%%%%%%%%%%%%%%%%%%%%%%%%%%%%%%%%%%%%%%%%%%%%%%%%

\section{Construction of the Maximal Function}
\label{sec:harmonic}

In this section, we use the covering lemma to generalize some results from the classical harmonic analysis to our situation. 
First, we confirm the existence of $\eta$-admissible cylinders centering almost everywhere under some assumptions on $u$. 
Then we prove the main theorem for the maximal function on these skewed cylinders and show related results similar to the classical case.

\subsection{Existence of Admissible Cylinders}

To begin with, we need some assumptions to guarantee the existence of $\eta$-admissible cylinders centering almost everywhere, which are the following. For the entire Section \ref{sec:harmonic}, we assume

\begin{assumption}
\label{assumption}
For some $1 \le p \le \infty$,
\begin{enumerate}[\upshape (1)]
    \item $\mmu \in L ^p ((S, T) \times \Rd)$.
    \item $\div u = 0$.
\end{enumerate}
\end{assumption}

\begin{proposition}
\label{prop:existence2}
Let $\eta > 0$. For almost every $\tx \in (S, T) \times \Rd$, $\Qe \tx$ is $\eta$-admissible for sufficiently small $\e$ (depending on $\tx$).
Moreover, we have
\begin{align*}
    \lim _{\e \to 0} \diam (\Qe \tx) = 0,
\end{align*}
where $\diam$ refers to the $(d + 1)$-dimensional diameter.
\end{proposition}

% \begin{remark}
% If $\grad u \in L ^\infty$, it is obvious that any $\Qe \tx$ will be $\eta$-admissible if $\e$ is small.
% \end{remark}

Before showing the proof of Proposition \ref{prop:existence2}, we first give a general lemma on the $L ^1$ boundedness of the map $f \mapsto f _\e$ defined below. Given $f \in L ^1 _{\loc} ((S, T) \times \Rd)$, for $x \in \Rd$, $t \in (S, T)$, $\e > 0$, we define 
\begin{align}
    \label{eqn:def-fe}
    f _\e \tx = \begin{cases}
        \fint _{\Qe \tx} f (s, y) \d y \d s & t \in (S + \e ^2, T - \e ^2) \\
        0 & t \in (S, S + \e ^2] \cup [T - \e ^2, T)
    \end{cases}.
\end{align}
Then we have the following bound on $f _\e$.

\begin{lemma}[$L ^1$ Boundedness]
\label{lem:L1-boundedness}
Given $f \in L ^1 ((S, T) \times \Rd)$, we have
% $f \mapsto f _\e$ is a bounded map from $L ^1 (\R ^{d+1})$ to $L ^1 (\R ^{d+1})$ uniformly in $\e$, in fact,
\begin{align*}
    \| f _\e \|_{\L1{(S, T) \times \Rd}} \le 
    \| f \|_{\L1{(S, T) \times \Rd}}.
\end{align*}
\end{lemma}

\begin{proof} 
    A direct computation gives
    \begin{align}
        \notag
        &\int _{S + \e ^2} ^{T - \e ^2} \int _{\Rd} |f _\e \tx| \d x \d t \\
        \notag
        &= \int _{S + \e ^2} ^{T - \e ^2} \int _{\Rd} \frac1{|\Qe|} \left| \int _{\Qe \tx} f (s, y) \d y \d s \right| \d x \d t \\
        % &= \int _{\RpRt} \frac1{|\Qe|} \left| \int _{\Qe \tx} f (s, y) \d y \d s \right| \d t \d x \\
        \notag
        & \le \frac1{|Q _\e|} 
        \int _S ^T \int _{\Rd}
        \int _{S + \e ^2} ^{T - \e ^2} \int _{\Rd} 
        |f (s, y)| \inds{(s, y) \in \Qe \tx} 
        \d x \d t \d y \d s \\
        \label{eqn:computing-cocylinder}
        & = \frac1{|Q _\e|} \int _S ^T \int _{\Rd} |f (s, y)| \mathcal L ^{d + 1} (\tilde Q _\e (s, y)) \d y \d s
    \end{align}
    where we define for any fixed $(s, y) \in (S, T) \times \Rd$ the \textbf{dual Lagrangian cylinder} by (see \cite{Isett2016} for a detailed discussion of these cylinders)
    \[
        \tilde Q _\e (s, y) := \left\{
            \tx \in (S + \e ^2, T - \e ^2) \times \Rd: (s, y) \in \Qe \tx 
        \right\}.
    \]
    Then from the definition of $\Qe \tx$, we can see that
    \begin{align*}
        \tilde Q _\e (s, y) 
        % &:= \left\{
        %     \tx: (s, y) \in \Qe \tx
        % \right\} \\
        &\subset \left\{
            \tx: |t - s| < \e ^2, |\Xe (t, x; s) - y| < \e
        \right\} \\
        & = \left\{
            \tx: |t - s| < \e ^2, x' := \Xe (t, x; s) \in B _\e (y)
        \right\} \\
        & = \left\{
            \tx: |t - s| < \e ^2, x' \in B _\e (y), x = \Xe (s, x'; t) 
        \right\}.
    \end{align*}
    Because $\ue$ is also divergence free, measure of a set is invariant under the flow, so we have
    \begin{align*}
        \mathcal L ^d \left(\{
            \Xe (s, x'; t): x' \in B _\e (y)
        \}\right) = \mathcal L ^d (B _\e (y)).
    \end{align*}
    Thus the measure of the dual cylinder is
    \begin{align*}
        \mathcal L ^{d + 1} (\tilde Q _\e (s, y)) &\le \mathcal L ^{d + 1} \left(\left\{
            \tx: |t - s| < \e ^2, x' \in B _\e (y), x = \Xe (s, x'; t) 
        \right\}\right) \\
        &= \int _{\max(S, s - \e ^2)} ^{\min(T, s + \e ^2)}
            \mathcal L ^d \left(\left\{
                \Xe (s, x'; t): x' \in B _\e (y)
            \right\}\right)
        \d t \\
        &\le 2\e ^2 |B _\e| = |\Qe|.
    \end{align*}
    Plugging into \eqref{eqn:computing-cocylinder}, we conclude that
    \begin{align*}
        \int _{S + \e ^2} ^{T - \e ^2} \int _{\Rd} |f _\e \tx| \d x \d t
        & \le \frac1{|Q _\e|} \int _S ^T \int _{\Rd} |f (s, y)| \mathcal L ^{d + 1} (\tilde Q _\e (s, y)) \d y \d s \\
        & \le \int _S ^T \int _{\Rd} |f \tx| \d x \d t.
    \end{align*}
    
\end{proof}

\begin{proof}[Proof of Proposition \ref{prop:existence2}]
If $p = \infty$ in the Assumption \ref{assumption}, the conclusions follow naturally from the Definition \ref{def:flow} and \ref{def:admissible-cylinder}, as now both the velocity field and its gradient are locally bounded. We shall only focus on the case $p < \infty$ from now.

Without loss of generality, assume $\eta \le \eta _0$. For $S < t < T$, $x \in \Rd$, define
\begin{align*}
    F \tx &:= [\mmut (x)] ^p \in L ^1 ((S, T) \times \Rd).
\end{align*}
$F _\e \tx$ is defined same as in \eqref{eqn:def-fe}.
% For $\e < \sqrt{\min \{t - S, T - t\}}$, define 
% \begin{align*}
%     F _\e \tx &:= \fint _{\Qe \tx} F (s, y) \d y \d s .
% \end{align*}
Lemma \ref{lem:L1-boundedness} shows that $\| F _\e \| _{L ^1} \le \| F \| _{L ^1}$. We want to show that for sufficiently small $\e$,
\begin{align*}
    F _\e (t, x) \le \eta ^p \e ^{-2p}.
\end{align*}
By Remark \ref{rem:jensen}, this implies that $\Qe \tx$ is $\eta$-admissible.
Define the set of non-admissible points by
\begin{align*}
    \Omega _\e = \left\lbrace
        (t, x) \in (S + \e ^2, T - \e ^2) \times \Rd: F _\e (t, x) > \eta ^p \e ^{-2p}
    \right\rbrace.
\end{align*}
By Chebyshev's inequality, its measure is bounded by
\begin{align*}
    |\Omega _\e| \le 
    |\{ F _\e > \eta ^p \e ^{-2p} \}|
    \le \frac{\|F _\e\| _{L ^1}}{\eta ^p \e ^{-2p}}
    \le \frac{\|F\| _{L ^1}}{\eta ^p}  \e ^{2p} \to 0
\end{align*}
as $\e \to 0$. Therefore, $|\cap _{\e > 0} \Omega _\e| = 0$, that is, the set of points at which no $\eta$-admissible cylinder centers has measure zero. In other words, for almost every point $(t, x)$, there exists $\e > 0$ such that $\Qe \tx$ is $\eta$-admissible.

This is not enough to show the conclusion, because $\Omega _\e$ may not be monotone in $\e$. To see that $\Qe \tx$ is $\eta$-admissible for all sufficiently small $\e$, let us define
\begin{align*}
    \Omega' _\e = \left\lbrace
        (t, x) \in (S + \e ^2, T - \e ^2) \times \Rd: F _\e (t, x) > \eta ^p (2 ^{d+1} \e) ^{-2p}
    \right\rbrace.
\end{align*}
Similar as before, Chebyshev's inequality implies
\begin{align*}
    |\Omega' _\e| 
    \le \frac{\|F\| _{L ^1}}{\eta ^p} (2 ^{d+1} \e) ^{2p}.
\end{align*}
In particular, for each $i \ge 1$, we have a geometric decaying upper bound as
\begin{align*}
    |\Omega' _{2 ^{-i}}| 
    \le \frac{\|F\| _{L ^1}}{\eta ^p} (2 ^{d+1} 2 ^{-i}) ^{2p}.
\end{align*}
It is a summable geometric series in $i$, thus by Borel-Cantelli lemma, we have
\begin{align*}
    \big|
        \limsup \nolimits_{i \to \infty} \Omega' _{2 ^{-i}}
    \big| 
    = \left|
        \bigcap \nolimits_{I > 0} \bigcup \nolimits_{i > I} \Omega' _{2 ^{-i}}
    \right|
    = 0.
\end{align*}
That is, for almost every $(t, x) \in (S, T) \times \Rd$, there exists $I > 0$ such that for all $i > I$, $(t, x) \notin \Omega' _{2 ^{-i}}$, i.e., for $\e _i = 2 ^{-i}$, we have
\begin{align*}
    F _{\e _i} (t, x) = \fint _{Q _{\e _i} \tx} F (s, y) \d y \d s \le \eta ^p (2 ^{d+1} \e _i) ^{-2p}.
\end{align*}
By Remark \ref{rem:jensen}, Jensen's inequality implies
\begin{align*}
    \e _i ^2 \fint _{Q _{\e _i} \tx} \mmu \d y \d s
    \le
    \e _i ^2 \left(
        \fint _{Q _{\e _i} \tx} [\mmu] ^p \d y \d s
    \right) ^\frac1p \le \frac\eta{4 ^{d+1}}.
\end{align*}
That is, $Q _{\e _i} \tx$ is $(4 ^{-d-1} \eta)$-admissible. 

We claim that if $Q _{\ea} \toxo$ is $(4 ^{-d-1} \eta)$-admissible, then for every $\eb$ within $\frac\ea4 \le \eb \le \frac\ea2$, $Q _{\eb} \toxo \subset \frac34 Q _{\ea} \toxo$. This can be proven by the claim
\begin{align}
    \label{eqn:claim}
    |\Xb \toxo[t] - \Xa \toxo[t]| \le \frac\ea4,
    % |X _{\eb} (t _0, x _0; t) - X _{\ea} (t _0, x _0; t)| \le \frac\e4 ,
    \qquad
    \text{for all } t \in (t _0 - \eb ^2, t _0 + \eb ^2)
\end{align}
whose proof is a slight modification of Lemma \ref{lem:eep}. Define $\Qa = Q _\ea \toxo$ and $\Qb = Q _\eb \toxo$. If we proceed the proof of Lemma \ref{lem:eep}, without knowing $\Qb$ is $\eta$-admissible, the only difficulty will arise at the last step \eqref{eqn:inclusion}, when we want to bound the integral of
$
    \eb ^{-d} \|\mmut \| _{\L1{\Bb(t)}}
$
in the Gr\"onwall's inequality. However, as long as \eqref{eqn:claim} holds at time $t$, $\Bb (t)$ is contained in $\Ba (t)$, thus 
\begin{align*}
    \eb ^{-d}\|\mmut \| _{\L1{\Bb(t)}} \le 4 ^d \ea ^{-d} \|\mmut \| _{\L1{\Ba(t)}}
\end{align*}
while the integral of the latter is bounded by $4 ^d \eta$. Following the same continuity argument we conclude \eqref{eqn:claim} in the end.
% \begin{align*}
%     &\int _{t _0} ^t \int _\ea ^\eb \frac{\partial}{\partial \e} \ue (s, \Xb (t _0, x _0; s)) \d\e \d s \\
%     &\qquad \le \int _{t _0} ^t \int _\ea ^\eb C \eb ^{-d} \|\mmus \| _{\L1{\Bb(s)}} \d\e \d s \\
%     &\qquad \le |\ea - \eb| C \eb ^{-d} \int _{t _0} ^t \|\mmus \| _{\L1{\Bb(s)}} \d s \\
%     &\qquad \le |\ea - \eb| C \eb ^{-d} \|\mmu \| _{\L1{\Qa}} \\
%     &\qquad < |\ea - \eb| C \eb ^{-d} \ea ^{d} 4 ^{-d-1} \eta \\
%     &\qquad \le |\ea - \eb| C \frac\eta4 \le \frac14 C \eta \ea \le \frac18 \ea
% \end{align*}
% as long as $\Bb$ does not exit $\Ba$. Following the rest of the proof, we conclude \eqref{eqn:claim} in the end. 

By this claim, for every $\e$ between $\frac{\e _i}{4}$ and $\frac{\e _i}{2}$, we have 
\begin{align*}
    Q _\e \tx \subset \frac34 Q _{\e _i} \tx \subset \left(\frac34\right) ^2 Q _{\e _{i - 1}} \tx \subset \cdots
\end{align*}
which implies $\diam (\Qe \tx) \to 0$ as $\e \to 0$. Although we do not have monotonicity for $Q _\e (t, x)$ in $\e$, we have this ``monotonicity with gaps''. Moreover, since $Q _\e \tx \subset Q _{\e _i} \tx$, $\e > \frac{\e _i}4$, we can bound $F _\e$ by
\begin{align*}
    F _\e \tx 
    = \frac1{|\Qe|} \int _{\Qe \tx} F \d y \d s 
    \le \frac{|Q _{\e _i}|}{|\Qe|} \fint _{Q _{\e _i} \tx} F \d y \d s 
    &\le 4 ^{d + 1} \eta ^p (2 ^{d+1} \e _i) ^{-2p} \\
    &\le \eta ^p \e _i ^{-2p} \le \eta ^p \e ^{-2p}.
\end{align*}
Thus $\tx \notin \Omega _\e$ for every $\e \in \left[\frac{\e _i}{4}, \frac{\e _i}{2} \right]$ and for every $i > I$, that is, for every $\e \le 2 ^{-I-1}$. This means $\Qe \tx$ is admissible for all $\e$ sufficiently small. 
\end{proof}

Following this existence proposition, we furthermore have the following corollary on the $L ^1$ convergence.
\begin{corollary}[$L ^1$ Convergence]
\label{cor:L1-convergence}
Let $f \in L ^1 ((S, T) \times \Rd)$, and define $f _\e$ by \eqref{eqn:def-fe}, then
\[
    f _\e \to f \text{ in } L ^1 ((S, T) \times \Rd) \text{ as } \e \to 0.
\]
\end{corollary}

\begin{proof}
For any $\delta > 0$, we can find $g \in \cci ((S, T) \times \Rd)$ such that $\|f - g\| _{L ^1} < \frac\delta3$. Denote $h = f - g$, then $\|h\| _{L ^1} < \frac\delta3$, and by Lemma \ref{lem:L1-boundedness}, also $\|h _\e\| _{L ^1} < \frac\delta3$ (we define $h _\e$ in the same way as \eqref{eqn:def-fe}). Since $g$ is uniformly continuous, it is clear that as $\diam (\Qe \tx) \to 0$,
\begin{align*}
    \|g - g _\e\| _{L ^1} \le \int _{(S + \e ^2, T - \e ^2) \times \Rd} \fint _{\Qe \tx} |g(t, x) - g(s, y)| \d y \d s \d x \d t < \frac\delta3
\end{align*}
for sufficiently small $\e$ such that $g (t, \cdot) = 0$ in $(S, S + \e ^2) \cup (T - \e ^2, T)$. Thus
\begin{align*}
    \|f - f _\e\| _{L ^1} = \|g + h - g _\e - h _\e\| _{L ^1} \le \|g - g _\e\| _{L ^1} + \|h\| _{L ^1} + \|h _\e\| _{L ^1} < \delta
\end{align*}
provided $\e$ is small enough.
\end{proof}

\subsection{Maximal Function}
The Existence Proposition \ref{prop:existence2} ensures the maximal function is well-defined almost everywhere. With the help of covering lemma, we can prove the bounds for the maximal function. A lot of ideas are borrowed from \cite{stein1970}. We do not claim any originality for results in this section, but only put them here for the sake of completeness.

\begin{proof}[Proof of Theorem \ref{thm:maximal-function}]
By the Existence Proposition \ref{prop:existence2}, for almost every $\tx \in (S, T) \times \Rd$, the set $\{\e > 0: \Qe \tx \text{ is $\eta$ admissible}\}$ is nonempty, so the maximal function $\mmq (f)$ is well-defined almost everywhere.
\begin{enumerate}[(1)]
    \item
        This is evident from the definition, since for any $\tx$ it holds that
        \begin{align*}
            \fint _{\Qe \tx} |f (s, y)| \d y \d s \le \|f\| _{L ^\infty}.
        \end{align*}
    \item 
    For any $\lambda > 0$, let $E _\lambda = \left\lbrace \tx: (\mmq f)\tx > \lambda \right\rbrace$ be the superlevel set. Then by definition, there is an $\eta$-admissible $\Qe$ centered at each point $\tx \in E _\lambda$, such that
    \begin{align*}
        |\Qe|< \frac1{\lambda} \int _{\Qe \tx} |f (s, y)| \d y \d s.
    \end{align*}
    Their radii are thus uniformly bounded. Thanks to the Covering Lemma Proposition \ref{prop:covering-lemma}, we can choose a pairwise disjoint subcollection $\{Q _{\e _j} (t ^j, x ^j)\}$, such that
    \begin{align*}
        \sum _j |Q _{\e _j}| \ge \frac1C \left|
            \bigcup \nolimits _{\tx \in E _\lambda} \Qe \tx
        \right|.
    \end{align*}
    Therefore the measure of the superlever set can be bounded by
    \begin{align*}
        |E _\lambda| \le C \sum _j |Q _{\e _j}| \le \frac{C}{\lambda} \sum _j \int _{Q _{\e _j} (t ^j, x ^j)} |f| \d x \d t \le \frac{C}{\lambda} \int _{(S, T) \times \Rd}  |f| \d x \d t.
    \end{align*}
    
    \item For the type $(q, q)$ part, we use Marcinkiewicz interpolation. Note that $\mmq$ is subadditive: $\mmq (f + g) \le \mmq (f) + \mmq (g)$.
    % \begin{align*}
    %      \mmq (f + g) \tx 
    %     & = \sup _\e \fint _{\Qe \tx} |f + g| \d x \d t \\
    %     & \le \sup _\e \fint _{Q _\e \tx} |f|\d x \d t 
    %     + \sup _\e \fint _{Q _\e \tx} |g|\d x \d t\\
    %     &= \mmq (f) \tx + \mmq (g) \tx.
    % \end{align*}
    We can split $f = f _1 + f _2$ where $f _1 = f \chi _{|f| \le \frac\lambda2}$ and $f _2 = f \chi _{|f| > \frac\lambda2}$. First, the strong type $(\infty, \infty)$ estimate applied to $f _1$ yields
    \begin{align*}
        \|\mmq (f _1)\| _{L ^\infty} \le \frac\lambda 2.
    \end{align*}
    Thus we have
    \begin{align*}
        \mmq (f) \le \mmq (f _1) + \mmq (f _2) \le  \frac\lambda2 + \mmq (f _2).
    \end{align*}
    So $\mmq (f) > \lambda$ implies $\mmq (f _2) > \frac\lambda2$. Next, the weak type $(1, 1)$ estimate applied to $f _2$ yields
    \begin{align*}
        \mu (E _\lambda) \le \mu \left(\left\lbrace
            \mmq (f _2) > \frac\lambda 2
        \right\rbrace\right) \le \frac {2C}{\lambda} \| f _2 \| _{L ^1}.
    \end{align*}
    By the layer cake representation, we have that
    \begin{align*}
        \int _{(S, T) \times \Rd} [\mmq (f)] ^q \d t \d x &= q \int _0 ^\infty \mu (E _\lambda) \lambda ^{q - 1} \d \lambda \\
        & \le 2 C q \int _0 ^\infty \frac1\lambda \int _{(S, T) \times \Rd} |f| \chi _{|f| > \frac\lambda2} \lambda ^{q - 1} \d x \d t \d \lambda \\
        & = 2 C q \int _{(S, T) \times \Rd} |f| \int _0 ^{2|f|} \lambda ^{q - 2} \d \lambda \d x \d t \\
        & = \frac{2 C q \cdot 2^{q - 1}}{q - 1} \int _{(S, T) \times \Rd} |f| ^q \d x \d t = C _q \| f \| _{L ^q} ^q.
    \end{align*}
    This finishes the proof of the theorem.
\end{enumerate}
\end{proof}

This theorem, together with the $L ^1$ convergence will imply the almost everywhere convergence of $f _\e$.

\begin{corollary}[a.e. Convergence]
\label{cor:pointwise-convergence}

Given $f \in L ^1 _{\loc} ((S, T) \times \Rd)$, for almost every $(t, x) \in (S, T) \times \Rd$, we have $f _\e (t, x) \to f (t, x)$ as $\e \to 0$, where $f _\e$ is defined in \eqref{eqn:def-fe}.
\end{corollary}

\begin{proof}
According to the Proposition \ref{prop:existence2}, $\diam (\Qe \tx) \to 0$ for almost every $\tx$, so we can assume $f$ is compactly supported and thus integrable without loss of generality.
By Corollary \ref{cor:L1-convergence} $L ^1$ convergence, we can find a subsequence which converges to $f$ almost everywhere, hence it suffices to show the following oscillation function is zero almost everywhere: for $f \in L ^1 _{\loc} ((S, T) \times \Rd)$, define the oscillation function by
\begin{align*}
    \Omega f \tx = \limsup _{\e \to 0} f _\e \tx - \liminf _{\e \to 0} f _\e \tx.
\end{align*}
For a uniformly continuous function $g$, we have $\Omega g \equiv 0$ almost everywhere, again using the fact that $\diam (\Qe (t, x)) \to 0$ by Proposition \ref{prop:existence2}. Moreover, notice that as $\e \to 0$, $\Qe \tx$ is $\eta$-admissible, so we have
\begin{align*}
    \limsup _{\e \to 0} f _\e \tx \le \limsup _{\e \to 0} |f _\e \tx| \le \mmq (f) \tx, \\
    - \liminf _{\e \to 0} f _\e \tx \le \limsup _{\e \to 0} |f _\e \tx| \le \mmq (f) \tx,
\end{align*}
so $\Omega f \le 2\mmq (f)$ almost everywhere. Now we fix $\lambda > 0$. For any given $\delta > 0$, we split $f = g + h$ with $g \in \cci ((S, T) \times \Rd)$ and $\| h \| _{L ^1} < \delta$, we have 
\begin{align*}
    \Omega f \le \Omega h + \Omega g = \Omega h \le 2 \mmq (h).
\end{align*}
By Theorem \ref{thm:maximal-function}, weak type $(1, 1)$ estimate gives
\begin{align*}
    \mu (\{ \Omega f > \lambda \}) \le \mu \left(\left\{ \mmq (h) > \frac\lambda2 \right\}\right) \le \frac{2C}{\lambda} \|h\| _{L ^1} = \frac{2C}{\lambda}\delta.
\end{align*}
Set $\delta \to 0$ we obtain
\begin{align*}
    \mu (\{ \Omega f > \lambda \}) = 0.
\end{align*}
This is true for any $\lambda > 0$, therefore we actually have
\begin{align*}
    \mu (\{ \Omega f > 0 \}) = 0.
\end{align*}
This means, for almost every $\tx \in (S, T) \times \Rd$, the oscillation is zero and
\begin{align*}
     \limsup _{\e \to 0} f _\e \tx = \liminf _{\e \to 0} f _\e \tx = \lim _{\e \to 0} f _\e \tx = f \tx.
\end{align*}
\end{proof}

Using the definition of $\mmq$, it is easy to deduce the following.
\begin{corollary}
For $f \in L ^1 _{\loc} ((S, T) \times \Rd)$, $f \le \mmq (f)$ almost everywhere.
\end{corollary}

To conclude this section, we present a slightly stronger result than the almost everywhere convergence.

\begin{theorem}[$\mathcal{Q}$-Lebesgue Differentiation Theorem]
    Under the same assumption of Corollary \ref{cor:pointwise-convergence}, for almost every $\tx \in (S, T) \times \Rd$, we have
    \begin{align}
        \label{Q-lebesgue-point}
        \lim _{\e \to 0} \fint _{\Qe \tx} |f (s, y) - f \tx| \d y \d s = 0.
    \end{align}
\end{theorem}

If \eqref{Q-lebesgue-point} is true for $\tx$, we call it a \textbf{$\mathcal{Q}$-Lebesgue point of $f$}, and define \textbf{$\mathcal{Q}$-Lebesgue set of $f$} to be the set of all $\mathcal{Q}$-Lebesgue points of $f$.

\begin{proof}
Consider any rational number $q \in \mathbb{Q}$. Then $f - q \in L ^1 _{\mathrm{loc}}$, thus by Corollary \ref{cor:pointwise-convergence}, we have
\begin{align*}
    |f - q| _\e \tx = \fint _{\Qe \tx} |f - q| (s, y) \d y \d s \to |f - q| \tx, \qquad \text{a.e. as $\e \to 0$}.
\end{align*}
By taking a countable intersection over $q \in \mathbb{Q}$ of all the sets where the convergence $|f - q|_\e \to |f - q|$ happens, we have
\begin{align*}
    |f - q| _\e \tx \to |f - q| \tx, \qquad \text{a.e.  as $\e \to 0$ for all $q \in \mathbb{Q}$}.
\end{align*}
By the density of rational numbers, it holds that
\begin{align*}
    |f - r| _\e \tx \to |f - r| \tx, \qquad \text{a.e.  as $\e \to 0$ for all $r \in \R$}.
\end{align*}
In particular, letting $r = f (t, x)$ gives
\begin{align*}
    |f - f \tx| _\e \tx \to |f \tx - f \tx| = 0, \qquad \text{a.e. as $\e \to 0$}.
\end{align*}
This is equivalent to \eqref{Q-lebesgue-point}.
\end{proof}

\section{Application to the Navier-Stokes Equations}
\label{sec:navier-stokes}

In this section, we give an example of how to use the maximal function to bridge between the local study and global results. Here we provide an alternative proof for $L ^p$-weak integrability for higher derivatives of 3D Navier-Stokes equations. In Proposition 2.2 of \cite{Choi2014} (case $r = 0$), the authors obtained the following local theorem. Recall that $B _r$ represents a ball of radius $r$ in $\R ^3$.

\begin{proposition}[Choi \& Vasseur, 2014]
    \label{prop:choi}
    Let $\vp \in \cci (B _1)$ be radial, satisfying $0 \le \vp \le 1$, $\int \vp \d x = 1$, and $\vp \equiv 1$ on $B _\frac12$. There exists $\bar \eta > 0$, such that if $v, p \in C ^\infty ((-4, 0) \times \R ^3)$ is a solution to
    \begin{align*}
        \pt v + (v \cdot \grad) v + \grad p &= \La v, \qquad
        \div v = 0
    \end{align*}
    verifying both
    \begin{align*}
        &\int _{\R ^3} \vp (x) v (t, x) \d x = 0 \qquad \text{for almost every } t \in (-4, 0), \\
        &\int _{-4} ^0 \int _{B _2} \left(
            % |\mm (\mm (|\grad u|))| ^2
            |\mm (|\mm (\grad v)| ^q)| ^\frac2q
            + |\grad ^2 p|
            + \sum _{m = d} ^{d + 4} \sup _{\delta > 0} \left|
                (\grad ^{m - 1} h ^\alpha) _\delta 
                * \grad ^2 p
            \right|
        \right) \d x \d t \le \bar \eta
    \end{align*}
    for some integer $d \ge 1$, $\alpha \in [0, 2)$, $q = \frac{12}{\alpha + 6}$, and $(\grad ^m h ^\alpha) _\delta$ is defined by
    \begin{align*}
        h ^\alpha (x) &:= \frac{\vp \left(\frac x2\right) - \vp (x)}{|x| ^{3 + \alpha}}, \qquad
        (\grad ^m h ^\alpha) _\delta (x) := \frac1{\delta ^3} (\grad ^m h ^\alpha) \left( \frac x\delta \right),
    \end{align*}
    then
    \begin{align*}
        |(-\La) ^\frac\alpha2 \grad ^d v| \le C _{d, \alpha} \qquad \text{\upshape in } \left(-1/36, 0\right) \times B _\frac16 (0).
    \end{align*}
\end{proposition}

This local theorem aims to control the magnitude of the higher fractional derivatives using quantities involving $\grad v$ and $\grad ^2 p$. Indeed, $\iint |\grad v| ^2 \d x \d t$ and $\iint |\grad ^2 p| \d x \d t$ both have the best scaling of the equation, and it is not hard to see the integrand in Proposition \ref{prop:choi} has the same scaling. The average zero condition ensures that the velocity $v$ is small as well, so that the quadratic flux term $v \cdot \grad v$ is manageable in the parabolic regularization. Moreover, since the purpose is to control a nonlocal quantity $(-\La) ^\frac\alpha2 \grad ^d v$, we need to gather nonlocal information using the maximal function $\mm$ and the $\grad ^m h ^\alpha$-maximal function: $\sup _{\delta > 0} |(\grad ^m h ^\alpha) _\delta * \cdot|$.

Let $u$ be a smooth solution to the Navier-Stokes equations \eqref{eqn:navier-stokes} in $(0, T)$. Since we need to center the cylinders at the terminal time for the local study, let us change our notation, and redefine
\begin{align*}
    \Qe \tx := \left\lbrace
        (s, y) : t - \e ^2 < s < t, |y - \Xe (t, x; s)| < \e
    \right\rbrace
\end{align*}
based on the velocity field $u$, which has $L ^2$ gradient and divergence zero. Results for the covering lemma and the maximal function can all be applied to this family of skewed cylinders, as we are just re-centering. By Galilean transform, the previous local proposition implies the following in the global coordinates.

\begin{corollary}
\label{cor:local}
There exists $\bar \eta > 0$, such that if $u, P \in C ^\infty ((0, T) \times \R ^3)$ is a solution to \eqref{eqn:navier-stokes} verifying for some $(t, x) \in (0, T) \times \R ^3$, $\e < \frac12\sqrt t$,
\begin{align*}
    &\frac1\e \int _{Q _{2 \e} \tx} \left(
        % |\mm (\mm (|\grad u|))| ^2
        |\mm (|\mmu| ^q)| ^\frac2q
        + |\grad ^2 P|
        + \sum _{m = d} ^{d + 4} \sup _{\delta > 0} \left|
            (\grad ^{m - 1} h ^\alpha) _\delta 
            * \grad ^2 P
        \right|
    \right) \d x \d t \le \bar \eta,
\end{align*}
then
\begin{align*}
    |(-\La) ^\frac\alpha2 \grad ^d u| \tx \le \frac{C _{d, \alpha}}{\e ^{d + \alpha + 1}}.
\end{align*}
\end{corollary}

\begin{proof}
For a fixed $\tx$, denote $r (s) = t + \e ^2 s$ and $z (s) = \Xe \tx[r]$, then 
\begin{align*}
    \dot r = \e ^2, \qquad \dot z = \e ^2 u _\e (r, z).
\end{align*}
We define the following change of coordinates:
\begin{align*}
    v (s, y) &= \e u (r, z + \e y) - \e u _\e (r, z), \\
    p (s, y) &= \e ^2 P (r, z + \e y) + \e y \partial _s [u _\e (r, z)].
\end{align*}
% \begin{align*}
%     v (s, y) &= \e u (t + \e ^2 s, \Xe (t, x; t + \e ^2 s) + \e y) \\
%     &\qquad - \e (u * \phi _\e) (t + \e ^2 s, \Xe (t, x; t + \e ^2 s)), \\
%     p (s, y) &= \e ^2 P (t + \e ^2 s, \Xe (t, x; t + \e ^2 s) + \e y) \\
%     &\qquad + \e y \partial _s [
%         (u * \phi _\e) (t + \e ^2 s, \Xe (t, x; t + \e ^2 s))
%     ].
% \end{align*}
Then we have the following in the new variables 
\begin{align*}
    \partial _s v &= \e \dot r \partial _t u + \e \dot z \cdot \grad u - \e \dot r \partial _t u _\e - \e \dot z \cdot \grad u _\e \\
    &= \e ^3 \left( 
        \partial _t u + u _\e \cdot \grad u - \partial _t u _\e - u _\e \cdot \grad u _\e
    \right), \\
    v \cdot \grad _y v &= v \cdot \e ^2 \grad u = \e ^3 \left( 
        u \cdot \grad u - u _\e \cdot \grad u
    \right), \\
    \La _y v &= \e ^3 \La u.
\end{align*}
Combining these three, we obtain 
\begin{align*}
    \partial _s v + v \cdot \grad v - \La v &= \e ^3 \left(
        \partial _t u + u \cdot \grad u - \La u - (\partial _t + u _\e \cdot \grad) u _\e
    \right) \\
    &= -\e ^3 (\grad P + (\partial _t + u _\e \cdot \grad) u _\e).
\end{align*}
Moreover, since $\partial _s u _\e (r, z) = \dot r u _\e + \dot z \cdot \grad u _\e = \e ^2 (\partial _t + u _\e \cdot \grad) u _\e$, we have
\begin{align*}
    \grad _y p = \e ^2 \e \grad P + \e \partial _s (u _\e (r, z)) = \e ^3 \left(
        \grad P + (\partial _t + u _\e \cdot \grad) u _\e
    \right).
\end{align*}
Therefore, $(v, p)$ is also a solution to the Navier-Stokes equations. Now we check that $v, p$ satisfy the assumptions of the Proposition \ref{prop:choi}. First, since $t \in (4 \e ^2, T)$, we know that $(v, p)$ is a smooth solution for $s \in (-4, 0)$. Next, we can verify that
\begin{align*}
    \int _{\R ^3} \vp (y) v (s, y) \d y = \e \left( 
        \int _{\R ^3} \vp (y) u (r, z + \e y) \d y - u _\e (r, z)
    \right) = 0.
\end{align*}
For the last condition, the change of variable yields
\begin{align*}
    |\mm (| \mm (\grad v)| ^q)| ^\frac2q &= \e ^{4} |\mm (| \mm (\grad u)| ^q)| ^\frac2q, \\
    \grad ^2 p &= \e ^2 \cdot \e ^2 \grad ^2 P + 0 = \e ^4 \grad ^2 P, \\
    (\grad ^{m - 1} h ^\alpha) _\delta * \grad ^2 p &= \e ^4 (\grad ^{m - 1} h ^\alpha) _{\e \delta} * \grad ^2 P.
\end{align*}
Since $Q _{2 \e} (t, x) = \{ (r, z + \e y): (s, y) \in (-4, 0) \times B _2 \}$ has space-time dimension 5, the last condition of Proposition \ref{prop:choi} is verified. As a consequence, we can bound
\begin{align*}
    |(-\La) ^\frac\alpha2 \grad ^d v (s, y)| \le C _{d, \alpha} \qquad \text{ in } (-1/36, 0) \times B _{\frac16} (0).
\end{align*}
In particular, when $s = 0, y = 0$, we have 
\begin{align*}
    C _{d, \alpha} \ge |(-\La) ^\frac\alpha2 \grad ^d v (0, 0)| = \e ^{d + \alpha + 1} |(-\La) ^\frac\alpha2 \grad ^d u (t, x)|.
\end{align*}
\end{proof}

Based on this, let us prove Theorem \ref{thm:vasseur} using the maximal function $\mmq$.

\begin{proof}[Proof of Theorem \ref{thm:vasseur}]
Define
\begin{align*}
    F (t, x) = |\mm (|\mmu| ^q)| ^\frac2q (t, x)
    + |\grad ^2 P|
    + \sum _{m = d} ^{d + 4} \sup _{\delta > 0} \left|
        (\grad ^{m - 1} h ^\alpha) _\delta 
        * \grad ^2 P
    \right| (t, x).
\end{align*}
It is well-known that for the Navier-Stokes equations, smooth solutions satisfy the following energy inequality:
% \begin{align*}
\[
    \|\grad u\| ^2 _{L ^2 ((0, T) \times \R ^3)} \le \frac12 \|u _0\| ^2 _{L ^2 (\R ^3)}.
\]
By the boundedness of the spatial maximal function in $L ^2 (\R ^3)$ and in $L ^\frac2q (\R ^3)$, we have 
\[
    \| |\mm (|\mmu| ^q)| ^\frac2q \| _{L ^1} \le \| \grad u \| _{L ^2} ^2.
\]
Moreover, using $-\La P = \div (u \cdot \grad u) = \grad u _i \cdot \partial _{x _i} u$, by the compensated compactness (\cite{Coifman1993}), we bound
\[ 
    \|\grad ^2 P\| _{L ^1 (0, T; \mathcal H ^1 (\R ^3))} \le \|\grad u\| _{L ^2 ((0, T) \times \R ^3)} ^2.
\]
where $\mathcal H ^1$ is the Hardy space. It is continuously embedded in $L ^1$, and we can use the Hardy norm to bound the $\grad ^m h ^\alpha$-maximal function by
\[ 
    \left\| 
        \sup _{\delta > 0} \left|
        (\grad ^{m - 1} h ^\alpha) _\delta 
        * \grad ^2 P (t)
        \right| 
    \right\| _{L ^1 (\R ^3)} \le C _{m, \alpha} \|\grad ^2 P (t)\| _{\mathcal H ^1 (\R ^3)}.
\] 
Combining the above estimates, we conclude that
\begin{align}
    \label{eqn:F-L1}
    \| F \| _{L ^1 ((0, T) \times \R ^3)} \le C \|u _0\| _{L ^2 (\R ^3)} ^2.
\end{align}

Denote $\eta := \min \left\{\frac{\bar \eta}{|Q _1|}, (\eta _0) ^2 \right\}$, and for $(t, x) \in (0, T) \times \R ^3$ we define
\begin{align*}
    I (\e) &:= \e ^4 \fint _{\Qe \tx} F (s, y) \d y \d s = \frac1{\e|Q _1|} \int _{\Qe \tx} F (s, y) \d y \d s.
\end{align*}
For all the $\mathcal Q$-Lebesgue point $\tx$ of $F$, we claim that there exists a positive $\e = \etoxo$ such that one of the following two cases is true:
\begin{enumerate}[{\ttfamily C{a}se 1.}]
    \item $\etoxo < t ^{\frac12}$, and $I (\etoxo) = \eta$.
    \item $\etoxo = t ^{\frac12}$, and $I (\etoxo) \le \eta$.
\end{enumerate}
The reason is that $\lim _{\e \to 0} I (\e) = 0 ^4 F \tx = 0$,
% \begin{align*}
%     I (\e) = \e ^4 \fint _{\Qe \tx} F (s, y) \d y \d s \to 0 ^4 F \tx = 0,
% \end{align*}
and $I (\e)$ is clearly a continuous function of $\e$ when $\e > 0$. As $\e$ ranges from 0 to $t ^{\frac12}$, either $I (\e)$ reaches $\eta$ at some $\etoxo < t ^{\frac12}$ \texttt{(Case 1)}, or it remains smaller than $\eta$ until $\etoxo = t ^{\frac12}$ \texttt{(Case 2)}.

At this $\e = \etoxo$ level, because
\begin{align*}
    |\mmu| ^2 
    \le |\mm (|\mmu)| ^q)| ^\frac2q
    \le F,
\end{align*}
by Jensen we have
\begin{align*}
    \e ^4 \left(
        \fint _{\Qe \tx} |\mmu| \d y \d s
    \right) ^2 \le \e ^4 
        \fint _{\Qe \tx} |\mmu| ^2 \d y \d s
    % \e ^4 \int _{\Qe \tx} |\mm (|\grad u|)| ^2 \d x \d t 
    \le I (\e) \le \eta,
\end{align*}
which implies $\Qe \tx$ is actually $\sqrt \eta$-admissible. So when in \texttt{Case 1},
\begin{align*}
    \eta &= \e ^4 \fint _{\Qe \tx} F (s, y) \d y \d s 
    \le \e ^4 \mmq F \tx.
\end{align*}
Combining with \texttt{Case 2}, we conclude
\begin{align*}
    \etoxo ^{-4} \le \max \left\{
        t ^{-2}, \frac{\mmq F \tx}{\eta}
    \right\}.
\end{align*}
Moreover, because $|Q _1| \cdot I (\etoxo) \le \bar \eta$ in both cases, Corollary \ref{cor:local} claims that 
\begin{align*}
    |(-\La) ^\frac\alpha2 \grad ^d u| \tx &\le \frac{C _{d, \alpha}}{\etoxo ^{d + \alpha + 1}}, \\
    \Rightarrow f ^p \tx &\le C _{d, \alpha} \etoxo ^{-4} \le C _{d, \alpha} \max \left\{
        t ^{-2}, \frac{\mmq F \tx}{\eta}
    \right\}.
\end{align*}
Finally, because $\mmq$ is of weak type $(1, 1)$, $\| \mmq F \| _{L ^{1, \infty}} \le C \| F \| _{L ^1}$. Together with \eqref{eqn:F-L1} we complete the proof of the theorem.
\end{proof}

\bibliographystyle{alpha}
\bibliography{bib}
\end{document}